\documentclass[12pt,reqno]{amsart}

\usepackage{amsmath}
\usepackage{enumerate}
\usepackage{amsfonts}
\usepackage{amssymb}
\usepackage{amsthm}
\usepackage{bbm}
\usepackage{cite}
\usepackage{xcolor}
\usepackage{tabularx}
\usepackage{graphicx}
\usepackage{xr}
\usepackage{xspace}
\usepackage{thmtools}
\usepackage{thm-restate}
\usepackage{fourier}
\usepackage[multiple]{footmisc}
\usepackage{microtype}
\usepackage[tmargin=1.25in,bmargin=1.25in,lmargin=1.3in,rmargin=1.3in]{geometry}
\usepackage{hyperref}
\usepackage[textsize=small]{todonotes}
\usepackage{soul}
\definecolor{halfgray}{gray}{0.55}
\definecolor{webgreen}{rgb}{0,.5,0}
\definecolor{webbrown}{rgb}{.6,0,0}
\definecolor{Maroon}{cmyk}{0, 0.87, 0.68, 0.32}
\definecolor{RoyalBlue}{cmyk}{1, 0.50, 0, 0}
\definecolor{Black}{cmyk}{0, 0, 0, 0}

\hypersetup{
    colorlinks=true, linktocpage=true, pdfstartpage=1, pdfstartview=FitV,
    breaklinks=true, pdfpagemode=UseNone, pageanchor=true, pdfpagemode=UseOutlines,
    plainpages=false, bookmarksnumbered, bookmarksopen=true, bookmarksopenlevel=1,
    hypertexnames=true, pdfhighlight=/O,
    urlcolor=webbrown, linkcolor=RoyalBlue, citecolor=webgreen,
}

\declaretheorem[numberwithin=section]{theorem}
\declaretheorem[sibling=theorem, name=Lemma]{lemma}
\declaretheorem[numbered=no, name=Remark] {rk}

\numberwithin{equation}{section}

\newcommand{\smallfrac}{\tfrac}
\newcommand{\sss}           { \scriptscriptstyle }
\newcommand{\Var}       {{{\rm Var}_p}}
\newcommand{\e}{\mathrm{e}}
\newcommand{\indi}{\mathbb{1}}
\newcommand{\Ep}{\mathbb{E}_p}

\newcommand{\Zk}{Z_{\sss \ge k}}

\newcommand{\dmax}{\Delta_{\sss \mathrm{max}}}
\newcommand{\dbr}{\partial B(r)}
\newcommand{\A}{\Acal}
\newcommand{\mnot}{t_{\sss {\rm mix}}}
\newcommand{\tmix}{{t_{\sss {\rm mix}}}}
\newcommand{\p}{P}

\newcommand{\vep} {\varepsilon}
\DeclareMathAlphabet{\pazocal}{OMS}{zplm}{m}{n}
\newcommand{\Acal}  {\pazocal{A}}
\newcommand{\Bcal}  {\pazocal{B}}
\newcommand{\Ccal}  {\pazocal{C}}

\newcommand{\Ecal}  {\pazocal{E}}
\newcommand{\Fcal}  {\pazocal{F}}

\newcommand{\Hcal}  {\pazocal{H}}

\newcommand{\Jcal}  {\pazocal{J}}

\newcommand{\Lcal}  {\pazocal{L}}
\newcommand{\Mcal}  {\pazocal{M}}

\newcommand{\Vcal}  {\pazocal{V}}

\newcommand{\R}     {\mathbb{R}}

\renewcommand{\P}   {\mathbb{P}}

\newcommand{\E}     {\mathbb{E}}
\newcommand{\Z}     {\mathbb{Z}}

\newcommand{\Zd}    {\mathbb{Z}^d}

\newcommand{\Pp}        {\mathbb{P}_p}
\newcommand{\Ppc}       {\mathbb{P}_{p_c}}
\newcommand{\Epc}       {\mathbb{E}_{p_c}}

\newcommand{\off}{\hbox{ {\rm off} }}
\renewcommand{\and}{\hbox{ {\rm and} }}
\newcommand{\with}{\hbox{ {\rm with} }}
\newcommand{\F}{\Fcal}
\renewcommand{\p}{\mbox{\bf p}}

\newcommand{\Tmix}{T_{\sss \mathrm{mix}}}
\newcommand{\supa}{\sup_{A \subset \Ecal}\,}
\newcommand{\ule}{\underline{e}}
\newcommand{\ole}{\overline{e}}

\newcommand{\rad}{r'}

\DeclareMathOperator{\diam} {diam}
\newcommand{\Cmax} {{\Ccal_1}}

\def\arrowfillCS#1#2#3#4{
   \thickmuskip0mu\medmuskip\thickmuskip\thinmuskip\thickmuskip
   \relax#4#1\mkern-7mu
   \cleaders\hbox{$#4\mkern-2mu#2\mkern-2mu$}\hfill
   \mkern-7mu#3
}
\def\lrfill{\arrowfillCS\leftarrow\relbar\rightarrow\relax}

\makeatletter
\renewcommand*\env@matrix[1][\arraystretch]{%
  \edef\arraystretch{#1}%
  \hskip -\arraycolsep
  \let\@ifnextchar\new@ifnextchar
  \array{*\c@MaxMatrixCols c}}
\makeatother

\newcommand{\conn}      {\leftrightarrow}
\newcommand{\lrr}	    {\stackrel{\le r}{\lrfill}}
\newcommand{\lrre}	    {\stackrel{=r}{\lrfill}}

\begin{document}
\title{Slightly subcritical hypercube~percolation}
\author{Tim Hulshof}
\address{Department of Mathematics and Computer Science, Eindhoven University of Technology, PO Box 513, 5600 MB Eindhoven, the Netherlands.}
\email{w.j.t.hulshof@tue.nl}
\author{Asaf Nachmias}
\address{Department of Mathematical Sciences, Tel Aviv University, Tel Aviv 69978, Israel.}
\email{asafnach@post.tau.ac.il}
\date{\today}
\begin{abstract}
We study bond percolation on the hypercube $\{0,1\}^m$
in the slightly subcritical regime where $p = p_c (1-\vep_m)$ and $\vep_m = o(1)$ but $\vep_m \gg 2^{-m/3}$ and study the clusters of largest volume and diameter. We establish that with high probability the largest component has cardinality $\Theta\left(\vep_m^{-2} \log(\vep_m^3 2^m)\right)$, that the maximal diameter of all clusters is $(1+o(1)) \vep_m^{-1} \log(\vep_m^3 2^m)$, and that the maximal mixing time of all clusters is $\Theta\left(\vep_m^{-3} \log^2(\vep_m^3 2^m)\right)$.

These results hold in different levels of generality, and in particular, some of the estimates hold for various classes of graphs such as high-dimensional tori, expanders of high degree and girth, products of complete graphs, and infinite lattices in high dimensions.
\end{abstract}
\maketitle

\vspace{1em}
{\small

}
\vspace{1em}
\hrule
\vspace{.5em}

\section{Introduction \& main results}
The \emph{hypercube} $Q_m$ is the graph with vertex set $\{0,1\}^m$ such that any two vertices of Hamming distance $1$ form an edge. We consider \emph{bond percolation} on it, that is, the random subgraph of $Q_m$ obtained by  independently removing each edge with probability $1-p \in [0,1]$ and retaining it otherwise. See Figure \ref{fig:hypercube} for an illustration.

\begin{figure}[t]
\includegraphics[width =.49\textwidth]{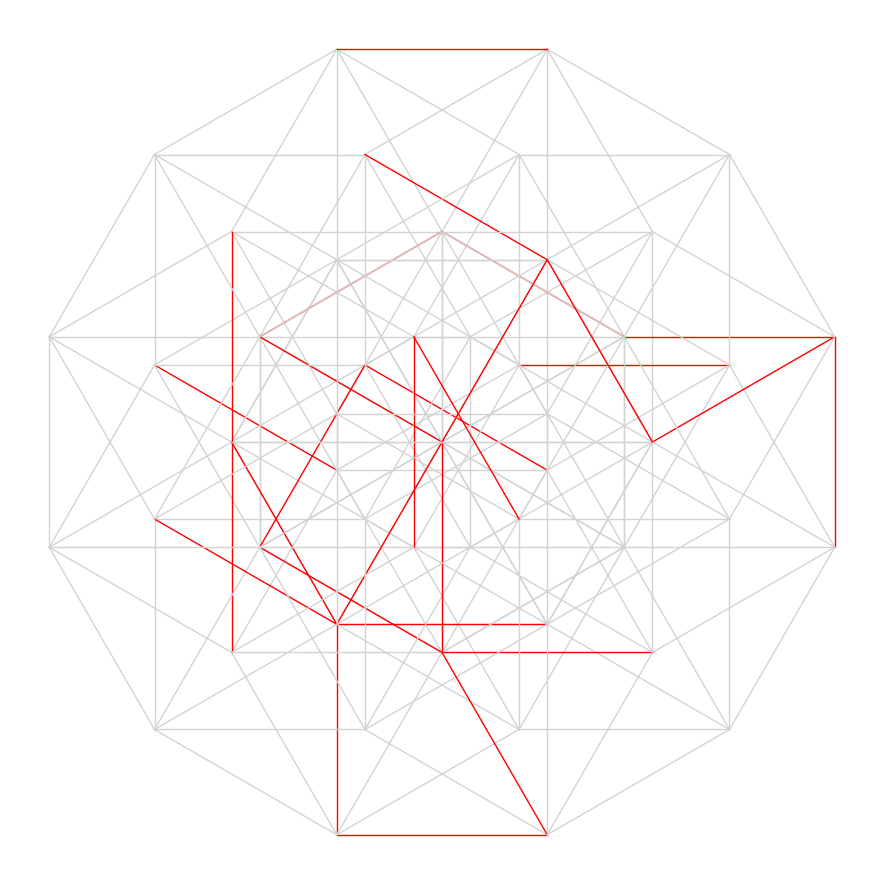}
\includegraphics[width =.49\textwidth]{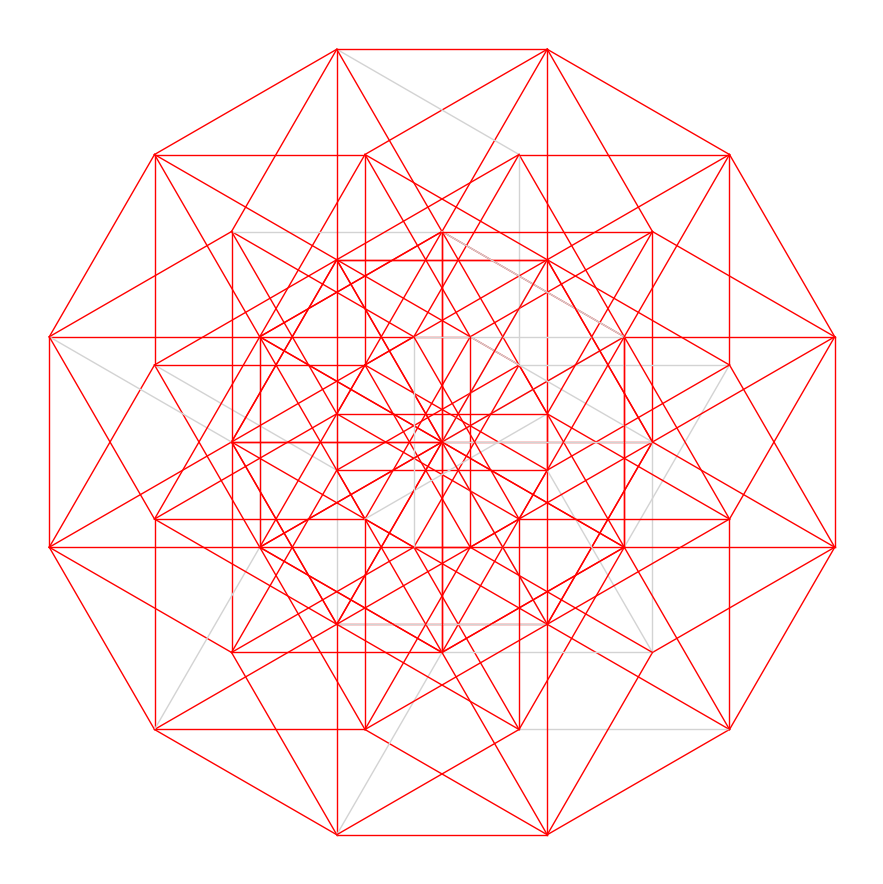}
\caption{Two realizations of percolation on $Q_6$ with $p=0.13$ (left) and $p = 0.87$ (right). The open edges are colored red.}\label{fig:hypercube}
\end{figure}

Hypercube percolation was introduced by Erd\H{o}s and Spencer \cite{ErdSpe79} and is compared there with the \emph{Erd\H{o}s-R\'enyi random graph} (ERRG) $G(n,p)$, which is bond percolation on the complete graph $K_n$ with percolation parameter $p$. Erd\H{o}s and Spencer investigated how the geometry of the hypercube affects the geometry of the percolation clusters and speculated that around the \emph{critical} probability hypercube percolation and the ERRG behave qualitatively alike.

Theirs and subsequent investigations (e.g.\ \cite{Boll82, Lucz90, JanKnuLucPit93,Aldo97} for the ERRG, and \cite{AjtKomSze82,BolKohLuc92,BorChaHofSlaSpe05a, BorChaHofSlaSpe05b, BorChaHofSlaSpe06, HofNac12} for hypercube percolation) confirm that this speculation holds to a very high degree. In fact, with various levels of success, this paradigm holds not just for the hypercube but for many classes of ``high-dimensional'' graphs. In other words, the behavior of the ERRG is \emph{universal}.

Due to the ERRG's complete symmetry, one can employ combinatorial arguments and branching processes comparisons to study it with great precision. Near the critical probability, these methods tend to fail in the presence of
geometry, even very simple geometries, such as the hypercube's.
Finding arguments that work in greater generality is the main challenge and motivation for studying hypercube percolation.

To understand the context of our results, it helps to first discuss the behavior of the ERRG. We put $p= c/n$ for some constant $c$, and write $\Ccal_j$ for the $j$th largest connected component of $G(n,p)$. It holds that when $c<1$ we have $|\Ccal_1| = \Theta(\log n)$ whp,\footnote{For a sequence of random variables $\{X_n\}$ and a function $f(n)$ we write $X_n  = \Theta(f(n))$ \emph{with high probability} (or whp) if there exist constants $C \ge c > 0$ such that $\lim_{n \to \infty}\P(c f(n) \le X_n \le C f(n)) = 1.$} while when $c>1$ we have that $|\Ccal_1| = \Theta(n)$ whp\cite{ErdRen60}.
Bollob\'as \cite{Boll82} was the first to study the delicate features of this transition that become apparent when, instead of keeping $c$ constant, we allow $c$ to depend on $n$ and let $c \to 1$ as $n \to \infty$. That and subsequent papers \cite{Boll82, Lucz90, JanKnuLucPit93,Aldo97} led to the following intricate picture:

Let $\vep_n =o(1)$ be a non-negative sequence. We can distinguish the following three regimes of the phase transition:
\begin{itemize}
	\item \emph{The slightly subcritical regime:} if \footnote{For two positive sequences $a_n$ and $b_n$ we write $a_n \gg b_n$ when $a_n/b_n \to \infty$.} $\vep_n \gg n^{-1/3}$ and $p =(1-\vep_n)/n$, we have for all fixed $j \ge1$ that
	\[
		\frac{|\Ccal_j|}{2 \vep_n^{-2} \log(\vep_n^3 n)} \stackrel{\P}{\longrightarrow} 1.
	\]
	\item \emph{The critical window:} if $\vep_n = a n^{-1/3}$ for some fixed $a\in\R$ and $p = (1 \pm \vep_n)/n$, we have for all fixed $j \ge 1$ that
	\[
		\left(\frac{|\Ccal_1|}{n^{2/3}}, \dots , \frac{|\Ccal_j|}{n^{2/3}}\right) \stackrel{\mathrm{d}}{\longrightarrow} (\chi_1, \dots, \chi_j ),
	\]
	for some sequence of random variables $(\chi_i)_{i=1}^j$ (that depends on the constant~$a$) supported on $[0,\infty)$.
	\item \emph{The slightly supercritical regime:} if $\vep_n \gg n^{-1/3}$ and $p =(1+ \vep_n)/n$, we have for $j \ge 2$,
	\[
		\frac{|\Ccal_1|}{2 \vep_n n} \stackrel{\P}{\longrightarrow} 1, \qquad \text{ and } \qquad
		\frac{|\Ccal_j|}{2 \vep_n^{-2} \log(\vep_n^3 n)} \stackrel{\P}{\longrightarrow} 1.
	\]
\end{itemize}

It has been shown so far that many of the features of the ERRG phase transition also hold for hypercube percolation. To state these results we first need to discuss the percolation threshold probability.
We write $V:=2^m$ for the number of vertices of $Q_m$ and $\Ccal(x)$ for the vertex set of the connected component of the vertex $x$, that is,
\[
	\Ccal(x) := \{y \in \{0,1\}^m \,:\, x \conn y\},
\]
where $\{x \conn y\}$ denotes the event that the vertices $x$ and $y$ are connected by a path of open edges in the percolation configuration (with the convention that $x \conn x$ for all $x$).\footnote{Below we will also use this notation to refer to the connected components of other graphs, and sometimes, we abuse notation and write $\Ccal(x)$ for the subgraph that is induced by this vertex set.} Since the hypercube is a transitive graph, the distribution of $\Ccal(x)$ as an unlabelled finite rooted graph is independent of our choice of $x$, so we will often consider $\Ccal(x)$ for some $x$ but simply write $\Ccal$.

We define the \emph{susceptibility} $\chi(p): = \Ep[|\Ccal|]$ and the \emph{critical parameter} $p_c =p_c(Q_m) \in [0,1]$ as the unique solution to
\begin{equation}\label{e:chidef}
	\chi(p_c) = \lambda V^{1/3}
\end{equation}
for some fixed $\lambda \in (0,\infty)$. There is some freedom in the choice of $\lambda$, see \cite{BorChaHofSlaSpe05a} for a detailed explanation. We can make sense of this definition via comparison with the ERRG, where $\chi((1 \pm \vep)/n) = \Theta(n^{1/3})$ if and only if $\vep = O(n^{-1/3})$ \cite{Boll82,Lucz90}.

Although the state of affairs for hypercube percolation is not nearly as complete as that of the ERRG, a rather thorough investigation is performed in \cite{AjtKomSze82,BolKohLuc92,BorChaHofSlaSpe05a,BorChaHofSlaSpe05b,BorChaHofSlaSpe06,HofSla06,HofNac12,HofNac14}. It is established there that there exists a critical window of width $O(V^{-1/3})$ around $p_c$ in which the largest components are of order $V^{2/3}$ and their distribution is not concentrated. The value of $p_c$ has been estimated using the lace expansion \cite{HofSla06} to be
\begin{equation}\label{e:pcestimate}
p_c(Q_m) = {1 \over m-1} + O(m^{-3})
\end{equation}
(see also \cite{HofNac14} for an elementary proof).

\subsection{The maximal volume of clusters in the slightly subcritical regime}
The first result in this regime is due to Bollob\'as, Kohayakawa, and {\L}uczak \cite{BolKohLuc92}, who showed that for percolation on $Q_m$ with $p = p_c(1-\vep_m)$,
\[
		\frac{|\Ccal_1|}{2 \vep_m^{-2} \log(\vep_m^3 V)} \stackrel{\P}{\longrightarrow} 1 \quad \text{ if } \quad \vep_m \ge \frac{(\log\log V)^2}{\sqrt{\log V} \log \log \log V}.
\]
Note that this constraint on $\vep_m$ is very far from $\vep_m \gg V^{-1/3}$, so this bound does not hold all the way up to the critical window.
The second bound, due to Borgs \emph{et al.} \cite{BorChaHofSlaSpe05a, BorChaHofSlaSpe05b} is that for any fixed $\delta>0$,
\begin{equation}\label{e:largeclusterupper}
	\Pp\left(\frac{\vep_m^{-2}}{3600} \le |\Ccal_1| \le (2+\delta) \vep_m^{-2} \log(\vep_m^3 V)\right) = 1-o(1) \quad \text{ if } \quad \vep_m \gg V^{-1/3}.
\end{equation}
So their bound works all the way up to the critical window, but the lower bound is not of the same order as the upper bound. The first result of the current paper is to prove the correct order of magnitude for the largest component throughout the entire slightly subcritical regime.

%%%%%%%%%%%%%%%%%%%% THEOREM %%%%%%%%%%%%%%%%%%%%%%%%%%%%%%%%%
\begin{theorem}\label{thm:numberlarge}
Consider percolation on $Q_m$ with $p = p_c (1-\vep_m)$ where $(\vep_m)$ is a non-negative sequence with $\vep_m \to 0$ and $\vep_m^3 V \to \infty$ as $m \to \infty$. Then, for any fixed $\alpha < 1$ and any $(j_m)$ such that $j_m \in [1, (\vep_m^3 V)^\alpha]$ we have
\[
	|\Ccal_{j_m}|  = \Theta\left(\vep_m^{-2} \log(\vep_m^3 V)\right) \text{ with high probability}.
\]
\end{theorem}
%%%%%%%%%%%%%%%%%%%%%%%%%%%%%%%%%%%%%%%%%%%%%%%%%%%%%%%%%%%
The upper bound is an immediate consequence of \eqref{e:largeclusterupper} and in Section \ref{sec:mainproof} we provide the corresponding lower bound. The proof of Theorem \ref{thm:numberlarge} relies on the following lower bound on the tail of $|\Ccal|$.
%%%%%%%%%%%%%%%%%%%% THEOREM %%%%%%%%%%%%%%%%%%%%%%%%%%%%%%%%%
\begin{theorem}\label{thm:lowertail} Consider percolation on $Q_m$ with $p = p_c(1-\vep_m)$ where $(\vep_m)$ is a non-negative sequence with $\vep_m \to 0$ and $\vep_m^3 V \to \infty$ as $m \to \infty$. Then there exists constants $c, c_a>0$ such that for any $A \in [1, c\vep V_m^{1/3}]$,
\begin{equation}\label{e:lowertail}
	\Pp \big(|\Ccal| \ge A \vep_m^{-2} \big) \ge \frac{c \vep_m}{A} \e^{-c_a A}.
\end{equation}	
\end{theorem}

\begin{rk}\emph{We believe that $|\Ccal_1|=(2+o(1)) \vep_m^{-2} \log(\vep_m^3 V))$ whp, as expected from the comparison with ERRG (and from \cite[Conjecture 3.2]{BorChaHofSlaSpe06}). This would follow if the constant $c_a$ in the latter theorem could be taken to be $1/2-o(1)$, but we are unable to prove this. See Section \ref{sec:about} for further discussion.}
\end{rk}

%%%%%%%%%%%%%%%%%%%%%%%%%%%%%%%%%%%%%%%%%%%%%%%%%%%%%%%%%%%

\subsection{The maximal diameter, one-arm probability, and mixing time} The \emph{diameter} of a finite connected graph is the largest graph distance between any two vertices. We define the maximal diameter $\dmax$ as the largest diameter among all the connected components.
In \cite{Lucz98} {\L}uczak shows that the slightly subcritical ERRG, that is, $G(n,p)$ with $p=(1-\vep_n)/n$ where $\vep_n = o(1)$ but $\vep_n \gg n^{-1/3}$, satisfies
 \[
 	\dmax = (1 \pm o(1)) \vep_n^{-1} \log(\vep_n^3 n) \, \qquad \text{with high probability.}
\]
Our second result in this paper is the analogous statement for the hypercube.
%%%%%%%%%%%%%%%%%%% THEOREM %%%%%%%%%%%%%%%%%%%%%%%%%%%%%%%%%%
\begin{theorem}\label{thm:maxdiam} Consider percolation on $Q_m$ with $p = p_c (1-\vep_m)$ where $(\vep_m)$ is a non-negative sequence with $\vep_m \to 0$ and $\vep_m^3 V \to \infty$ as $m \to \infty$. Then
\[
	\dmax = (1 \pm  o(1))(\vep_m^{-1} \log(\vep_m^3 V)) \text{  with high probability.}
\]
\end{theorem}
%%%%%%%%%%%%%%%%%%%%%%%%%%%%%%%%%%%%%%%%%%%%%%%%%%%%%%%%%%%

\begin{rk}\emph{Note that in the subcritical phase of the ERRG the largest cluster $\Cmax$ is \emph{not} the cluster with the largest diameter whp. Indeed, one can readily show that $\Cmax$ is a tree whp, and thus, if we condition on $|\Cmax|$, any tree with $|\Cmax|$ vertices has the same probability of being $\Cmax$. Since a uniformly chosen tree on $k$ vertices has diameter $\Theta(\sqrt{k})$ whp \cite{RenSze67}, we conclude that $\diam(\Cmax) = \Theta(\vep^{-1} \sqrt{\log(\vep^3 n)})$ whp, a factor $\sqrt{\log(\vep^3 n)}$ away from the maximal diameter. We expect this to hold in subcritical hypercube percolation but we were unable to prove this, see Section~\ref{sec:about} below.}
\end{rk}

The main ingredients in the proof of Theorem \ref{thm:maxdiam} are the following sharp bounds on the slightly subcritical boundary volume and one-arm probability:
\begin{theorem}\label{thm:Qmonearm}
Consider percolation on $Q_m$ with $p = p_c (1-\vep_m)$ where $(\vep_m)$ is a non-negative sequence with $\vep_m \to 0$ and $\vep_m^3 V \to \infty$ as $m \to \infty$. Then there exist $C,c>0$ such that for all integers $r = O(\vep_m^{-1} \log(\vep_m^3 V))$,
\[
	c (1-\vep_m)^r \le \Ep[|\partial B(r)|] \le C(1-\vep_m)^r,
\]
and
\[
	\frac{c}{r} (1-\vep_m)^r \le \Pp(\partial B(r) \neq \emptyset) \le \begin{cases} C r^{-1} & \text{ if } r \le \vep_m^{-1},\\
	C \vep_m (1-\vep_m)^r & \text{ if } r > \vep_m^{-1}.
	\end{cases}
\]
\end{theorem}

We next turn to analyzing the \emph{mixing time} of simple random walk on the clusters. Recall that the \emph{total variation distance} between two probability measures $\psi$ and $\phi$ on a finite set $\Vcal$ is defined as
\[
	\|\psi - \phi\|_{\sss {\rm TV}} := \sup_{S \subset \Vcal}|\psi(S) - \phi(S)| = \smallfrac12 \sum_{v \in \Vcal}|\psi(v) -\phi(v)|.
\]
Given a graph $G = (\Vcal, \Ecal)$, let $\pi(x)$ be the stationary distribution of the simple random walk on it, i.e., $\pi(x) = \deg(x) / (2|\Ecal|)$ and let $S_t$ be the lazy simple random walk on $G$ (i.e., a discrete time simple random walk that at each time step with probability $\smallfrac12$ stays put and otherwise jumps to a uniformly chosen neighbor). The \emph{mixing time} of the lazy random walk on $G$ is defined by
\begin{equation}\label{e:Tmixdef}
	T_{\sss {\rm mix}}(G) = T_{\sss {\rm mix}} \left(G; \smallfrac14\right) := \min \big\{ t \, : \, \max_{v \in \Vcal} \| \P(S_t \in \, \cdot \, \mid S_0 = v) - \pi\|_{\sss {\rm TV}} \le \smallfrac14\big\}
\end{equation}
(the choice of $\smallfrac14$ is standard and inessential, see \cite{LevPerWil09}). The mixing time thus describes the time at which the random walk's distribution first comes ``close'' to the stationary distribution in total variation. As usual, let $\vep_n$ be a non-negative sequence such that $\vep_n = o(1)$ and $\vep_n \gg n^{-1/3}$, and write $\Ccal^\star$ for the component of $G(n, (1-\vep_n)/n)$ with the largest mixing time. Ding, Lubetzky, and Peres \cite[Theorem 2]{DinLubPer12} proved that
\[
	T_{\sss {\rm mix}}(\Ccal^\star)= \Theta \big(\vep_n^{-3} \log (\vep_n^3 n)^2\big) \quad \text{ with high probability.}
\]
Our final result for the hypercube is almost the analogous statement.
\begin{theorem}\label{thm:mixing} Consider percolation on $Q_m$ with $p = p_c(1-\vep_m)$ where $(\vep_m)$ is a non-negative sequence with $\vep_m \to 0$ and $\vep_m^3 V \to \infty$ as $m \to \infty$. Let $\Ccal^\star$ be the component with the largest mixing time. Then there exist $C>0$ such that for any sequence $(\omega_m)$ with $\omega_m \to 0$ as $m \to \infty$,
\[
 	\omega_m \vep_m^{-3} \log (\vep_m^3 V)^2 \le T_{\sss {\rm mix}}(\Ccal^\star) \le C \vep_m^{-3} \log(\vep_m^3 V)^2 \quad \text{ with high probability.}
\]
\end{theorem}

\subsection{About the proofs.}\label{sec:about} Our proofs use the many tools and techniques developed in \cite{BorChaHofSlaSpe05a,HofNac12,KozNac09,NacPer08} to study the volume, diameter and mixing time of large clusters in critical percolation. The main new ingredients we develop in this paper, which can be seen as further developments to the aforementioned papers, are bounds on the moments of clusters conditioned on having large diameter (see Section~\ref{sec:mainproof}), sharp estimates for the one-arm event (Theorem \ref{thm:Qmonearm} and Theorems \ref{thm:subvolume} and \ref{thm:subdiam} below), and an estimate showing that the probability of a long arm event cannot be increased significantly by removing a small number of edges from the graph (Theorem \ref{thm:subdiamoffA}).

Our methods have an inherent limitation that prevents us from obtaining much sharper results on the volume of the largest cluster. In particular, we are unable to prove that the largest cluster is of size $(2+o(1))\vep^{-2} \log(\vep^3 V)$ whp. The limitation stems from the fact that in the subcritical phase there are many clusters of volume comparable to the largest one that exhibit many different geometries. The triangle condition \eqref{e:strongtriangle} below gives us a firm understanding of the event that the cluster has a large diameter, but less so on the event that its volume is large. Thus, the lower bound on $|\Ccal_j|$ obtained in Theorem~\ref{thm:numberlarge} is obtained by showing that clusters of large diameter (that is, diameter of order $\vep^{-1} \log(\vep^3 V)$) exist and that such clusters typically have large volume, that is, volume of order $\vep^{-2} \log(\vep^3 V)$. Unfortunately, the leading constant for the volume for such ``long'' clusters is strictly smaller than $2$. In fact, the largest cluster is expected to have much smaller diameter, i.e., of order $\vep^{-1} \sqrt{\log (\vep^3 V)}$, as in the ERRG case.

\subsection{General theorems}
Theorems \ref{thm:numberlarge}--\ref{thm:mixing} are stated for the hypercube $Q_m$, but the assertions there hold in various levels of generality. Theorems \ref{thm:numberlarge} and \ref{thm:lowertail} hold under the assumption of the triangle condition (see \cite{BarAiz91, BorChaHofSlaSpe05a} and \eqref{e:strongtriangle} below) and therefore hold, for instance, for the hypercube, finite tori $\Z_n^d$ with $d$ large but fixed, and to expander families of high degree and high girth. Theorem~\ref{thm:lowertail} even holds for \emph{infinite} graphs that satisfy the triangle condition of \cite{BarAiz91}. The bounds on the diameter, one-arm probability, and mixing time of Theorems \ref{thm:maxdiam}, \ref{thm:Qmonearm}, and \ref{thm:mixing} hold under the stronger assumptions of \cite[Theorem 1.3]{HofNac12}. We now describe these general conditions and state our most general theorems.

Given a graph $G$ and $p \in [0,1]$ we write $G_p$ for the random graph obtained from $G$ by performing bond-percolation on $G$ with parameter $p$ and denote by $\Pp$ this probability measure. We call the edges of $G_p$ \emph{open} and the edges not in $G_p$ \emph{closed}. For each vertex $x \in G$ we write $\Ccal(x)$ for the connected component of $x$ in $G_p$. Recall that we write $\chi(p) = \E_p[|\Ccal(x)]$, and that this quantity does not depend on our choice of $x$ when $G$ is transitive. For two vertices $x,y$ of $G$ we write $x \conn y$ for the event that there exists an open path in $G_p$ connecting $x$ to $y$.

In our general setting we are given a sequence of transitive graphs $(G_m)$ with vertex degree $m$ and the numbers $p_c(G_m)$ as defined in \eqref{e:chidef}. We write $V_m$ for the number of vertices in $G_m$. We are also given a sequence of nonnegative numbers $\vep_m$ satisfying $\vep_m=o(1)$ and $\vep_m^3 V_m \to \infty$. For ease of notation, we will often write $G$, $p_c$, $\vep$ and $V$ instead of $G_m$, $p_c(G_m)$, $\vep_m$ and $V_m$, respectively.

The \emph{triangle condition}, first defined in \cite{BarAiz91} and refined to the finite graph setting in \cite{BorChaHofSlaSpe05a}, is a certain condition on the sequence $(G_m)$ implying several results for the percolation phase transition. This is an extensively studied topic, see e.g.\ \cite{AizNew84, BarAiz91, HarSla90a, BorChaHofSlaSpe05a, BorChaHofSlaSpe05b, FitHof15, Kozm10, KozNac09, KozNac11, Scho01}.
We state here a useful variant of the triangle condition: the \emph{strong triangle condition} holds if there exists $C>0$ such that for any two vertices $x,y$ and any $p \leq p_c$ we have
\begin{equation}\label{e:strongtriangle}
	\sum_{u,v} \Pp(x \conn u) \Pp(u \conn v) \Pp(v \conn y)  = \mathbb{1}_{\{x=y\}} + C\big(\chi^3(p)/V + 1/m \big) \, .
\end{equation}
This condition has been verified for various classes of graphs, with the hypercube $Q_m$ and high-dimensional tori $\Z_n^d$ among them \cite{BorChaHofSlaSpe05b, HofNac14}. We now state our first result, generalizing Theorems \ref{thm:numberlarge} and \ref{thm:lowertail}.

\begin{theorem}\label{thm:general1}
Let $(G_m)$ be a sequence of finite transitive graphs satisfying the triangle condition \eqref{e:strongtriangle}. Consider percolation on $G_m$ with $p = p_c(1-\vep_m)$, where $(\vep_m)$ is a non-negative sequence with $\vep_m \to 0$ and $\vep_m^3 V_m \to \infty$ as $m \to \infty$.
Then the following assertions hold:
\begin{enumerate}
\item[(a)] For any fixed $\alpha < 1$ and any $(j_m)$ such that $j_m \in [1, (\vep_m^3 V_m)^\alpha]$ we have
\[
	|\Ccal_{j_m}|  = \Theta(\vep_m^{-2} \log(\vep_m^3 V_m)) \quad \text{ with high probability}.
\]
\item[(b)] There exist constants $c, c_a>0$ such that for any $A \in [1, c\vep_m V_m^{1/3}]$,
\begin{equation}
	\Pp(|\Ccal| \ge A \vep_m^{-2}) \ge \frac{c \vep_m}{A} \e^{-c_a A}.
\end{equation}	
\end{enumerate}
\end{theorem}

\begin{rk}\emph{A version of Theorem~\ref{thm:general1}(b) also holds for percolation on infinite lattices when the dimension is sufficiently large. In particular, our proof can be modified to show the analogous result when the infinite-lattice version of the triangle condition given by
\[
	 \sum_{x,y \in G}  \P_{p_c}(0 \conn x)\P_{p_c}(x \conn y)\P_{p_c}(y \conn 0) < \infty \, .
\]
This has been confirmed, among others, for nearest-neighbor percolation on $\mathbb{Z}^d$ when $d \ge 11$ \cite{FitHof15}, for certain ``finite-range spread-out'' percolation models on $\mathbb{Z}^d$ when $d > 6$ \cite{HarSla90a}, and for percolation on certain non-amenable Cayley graphs~\cite{Scho01, Scho02}. In this setting one can follow our proof --with straightforward modifications-- to conclude that there exist $c', c'_a>0$ such that for percolation at $p = p_c(1-\vep)$ and all $A \ge 1$,
\begin{equation}
	\Pp(|\Ccal| \ge A \vep^{-2}) \ge \frac{c' \vep}{A} \e^{-c'_a A}.
\end{equation}
}
\end{rk}

We now present the general version of Theorems \ref{thm:maxdiam}, \ref{thm:Qmonearm}, and \ref{thm:mixing}. Given a graph $G$, the $t$-step \emph{non-backtracking random walk} on $G$ starting from a vertex $x$ is a uniform measure on all paths $(X_1, \ldots, X_t)$ in $G$ such that $X_1=x$ and $X_{i} \neq X_{i-2}$ for all $3 \leq i \leq t$
(so the walk never backtracks). For two vertices $x,y$ of $G$ we write $\p^t(x,y)$ for the probability that a $t$-step non-backtracking random walk starting at $x$ ends at $y$. Given a connected graph $G$ and $\alpha \in (0,1)$ we define the
\emph{uniform non-backtracking mixing time} as
\begin{equation}\label{e:mnotdef}
	 \mnot := \mnot(G;\alpha) := \min \Big \{ t : \max_{x,y} {\p^t(x,y) + \p^{t+1}(x,y) \over 2}  \leq (1+\alpha)V^{-1} \Big \} \, .
\end{equation}
The averaging between $\p^t$ and $\p^{t+1}$ is incorporated to admit bipartite graphs, such as the hypercube, to the general setting. Note that although $\tmix$ is superficially similar to $\Tmix$, they are different quantities.

For later reference, we remark that Fitzner and van der Hofstad \cite[Theorem 3.5]{FitHof13} show that on the hypercube,
 \begin{equation}\label{e:Qmmix}
 	\mnot(Q_m; \alpha_m) = O(m \log m)
\end{equation}
for any $\alpha_m=o(1)$ that is at least polynomial, that is, $\alpha_m^{-1} = O(m^{c})$ for some fixed $c>0$.

The assumptions we make on the sequence $(G_m)$ of transitive graphs are that there exists a sequence $(\alpha_m)$ with $\alpha_m = o(1)$ and $\alpha_m\geq m^{-1}$ such that
\begin{equation}\label{e:assump0}
\mnot(G_m; \alpha_m) = o(V^{1/3}/\log^2(V)) \, ,
\end{equation}
and that
\begin{equation} \label{e:assump1}
(p_c (m-1))^{\mnot(G_m; \alpha_m)} = 1+O(\alpha_m) \, ,
\end{equation}
and that
\begin{equation} \label{e:assump2}
\max_{x,y} \sum_{u,v} \sum_{\substack{t_1, t_2, t_3 : \\t_1+t_2+t_3 \geq3}}^{\mnot(G_m ;\alpha_m)} \p^{t_1}(x,u)\p^{t_2}(u,v)\p^{t_3}(v,y) = O(\alpha_m / \log V) \, .
\end{equation}

We now state the general versions of Theorems \ref{thm:maxdiam}, \ref{thm:Qmonearm}, and \ref{thm:mixing}.

\begin{theorem}\label{thm:generalmaxdiam}
Let $(G_m)$ be a sequence of transitive graphs satisfying \eqref{e:assump0}, \eqref{e:assump1} and \eqref{e:assump2}. Consider percolation on $G_m$ with $p = p_c(1-\vep_m)$, where $(\vep_m)$ is a non-negative sequence with $\vep_m=o(1)$ and $\vep_m^3 V_m \to \infty$.
Then the following assertions hold:
\begin{enumerate}
\item[(a)]
\[
	 \dmax = (1 \pm o(1))(\vep_m^{-1} \log(\vep_m^3 V_m)) \text{  with high probability.}
\]

\item[(b)] There exist $c,C>0$ such that for all integers $r = O(\vep_m^{-1} \log(\vep_m^3 V_m))$,
\[
	c (1-\vep_m)^r \le \Ep[|\partial B(r)|] \le C(1-\vep_m)^r,
\]
and
\[
	\frac{c}{r} (1-\vep_m)^r \le \Pp(\partial B(r) \neq \emptyset) \le \begin{cases} C r^{-1} & \text{ if } r \le \vep_m^{-1},\\
	C \vep_m (1-\vep_m)^r & \text{ if } r > \vep_m^{-1}.
	\end{cases}
\]
\item[(c)] Let $\Ccal^\star$ be the component with the largest mixing time. Then there exist $C>0$ such that for any sequence $(\omega_m)$ with $\omega_m \to 0$ as $m \to \infty$,
\[
 	\omega_m \vep_m^{-3} \log (\vep_m^3 V_m)^2 \le T_{\sss {\rm mix}}(\Ccal^\star) \le C \vep_m^{-3} \log(\vep_m^3 V_m)^2 \quad \text{ with high probability.}
\]
\end{enumerate}
\end{theorem}

For percolation on the hypercube $Q_m$, assumptions \eqref{e:assump0} and \eqref{e:assump1} follow immediately from the estimates \eqref{e:pcestimate} and \eqref{e:Qmmix}. In \cite[Section 7.2]{HofNac12} it is shown that \eqref{e:assump2} holds for the hypercube. Hence, Theorem \ref{thm:generalmaxdiam} implies Theorems \ref{thm:maxdiam}, \ref{thm:Qmonearm}, and \ref{thm:mixing}.

Furthermore, assumptions \eqref{e:assump0}, \eqref{e:assump1}, \eqref{e:assump2} were verified in \cite[Theorem 1.4]{HofNac12} for expanders of high degree and high girth, hypercubes, and for products of complete graphs, and hence the conclusions Theorem \ref{thm:generalmaxdiam} hold for these classes of graphs as well. Lastly we remark that these assumption in fact imply the strong triangle condition \cite[Theorem~1.3(a)]{HofNac12}, but are not equivalent. Indeed, the tori $\mathbb{Z}^d_n$ when $n \to \infty$ and $d$ fixed satisfy \eqref{e:strongtriangle} but do not satisfy \eqref{e:assump1}.

\subsection{The structure of this paper}
In Section~\ref{sec:prelim} we start with some preliminaries: we recall bounds for subcritical and critical percolation from the literature, and we prove some easy consequences of these bounds. We also prove the (easy) lower bounds on the one-arm probability of Theorems~\ref{thm:Qmonearm} and \ref{thm:generalmaxdiam}(b).

In Section~\ref{sec:mainproof} we establish bounds on the moments of $|\Ccal|$ conditionally on having a large diameter, and use them to prove Theorems \ref{thm:numberlarge}, \ref{thm:lowertail}, and \ref{thm:general1}. In Section~\ref{sec:onearmsharp} we prove the upper bounds of Theorems~\ref{thm:Qmonearm} and \ref{thm:generalmaxdiam}(b), as well as Theorem~\ref{thm:subdiamoffA} concerning the effect that removing edges from the graph has on the one-arm probability.
In Section~\ref{sec:maxdiam} we then use these results to prove the bounds on the maximal diameter from Theorems~\ref{thm:maxdiam} and \ref{thm:generalmaxdiam}(a). Finally, in Section~\ref{sec:mixing} we prove the bounds on the mixing time from Theorem~\ref{thm:generalmaxdiam}(c) and Theorem~\ref{thm:mixing}.

\section{Preliminaries}\label{sec:prelim}

In this section we recall some of definitions, tools and previous results used in the proofs, and use them draw some simple conclusions. The first estimates involve the distribution of $|\Ccal(x)|$. Aizenman and Newman \cite[Proposition 5.1]{AizNew84} proved that if $G$ is a finite or infinite transitive graph,\footnote{The focus of \cite[Proposition 5.1]{AizNew84} is transitive \emph{infinite} graphs, but the statement and proof are valid for transitive finite graphs as well.} then for any $k \geq \chi(p)^2$
\begin{equation}\label{e:uppertail}
	\Pp(|\Ccal| \ge k) \le \sqrt{\frac{\e}{k}}\, \e^{-k / 2 \chi(p)^2}.
\end{equation}

Borgs \emph{et al.} \cite[Theorem 1.3]{BorChaHofSlaSpe05a} proved that if $G$ satisfies the strong triangle condition \eqref{e:strongtriangle}, then
\begin{equation}\label{e:chibds}
	\chi(p_c(1-\vep)) = \E_{p_c(1-\vep)}[|\Ccal|] =  (1 + o(1))\vep^{-1} .
\end{equation}

The following estimates concern the ``intrinsic'' metric of the percolation cluster, we require a few definitions first. Given vertices $x$ and $y$ and a non-negative integer $r$, we define the events
\begin{itemize}
	\item $\big\{x \stackrel{= r}{\lrfill} y\big\}$ if the \emph{shortest} path in $G_p$ connecting $x$ and $y$ has length precisely $r$,
	\item $\big\{x \stackrel{\le r}{\lrfill} y \big\}$ if the shortest path in $G_p$ connecting $x$ and $y$ has length at most~$r$,
	\item $\big\{x \stackrel{\ge r}{\lrfill} y \big\}$ if the shortest path in $G_p$ connecting $x$ and $y$ has length at least $r$.
\end{itemize}
It is worth noting here that $\{x \conn y\}$ and $\big\{x \stackrel{\le r}{\lrfill} y\big\}$
are monotone increasing with respect to adding edges (that is, if we replace a closed edge with an open edge, then the event continues to hold) while $\big\{x \stackrel{= r}{\lrfill} y \big\}$ and $\big\{x \stackrel{\ge r}{\lrfill} y \big\}$ are not (indeed, adding an edge to a graph can create a shorter shortest path between two vertices).

The \emph{intrinsic metric} ball of radius $r$ around a vertex $x$ in the graph $G$ and its boundary are defined by
\[
	B^{\sss G}_x (r) :=  \big\{y  \, : \, x \lrr y \big\} \qquad \text{ and } \qquad \partial B^{\sss G}_x (r) :=  \big\{y  \, : \, x \lrre y \big\} ,
\]
and we note that both are random sets with respect to $\Pp$. When $G$ is transitive we often abbreviate $B(r) = B^{\sss G}_x (r)$ and $\partial B(r) = \partial B^{\sss G}_x (r)$.

It is proved in \cite{KozNac09} that if $G$ satisfies the strong triangle condition \eqref{e:strongtriangle}, then there exists finite constants $C_1$ and $C_2$, that may depend on $\lambda$ of \eqref{e:chidef}, such that
\begin{equation}
     \E_{p_c} [|B(r)|] \le C_1 r, \label{e:critvolball}
   \end{equation}
   and
   \begin{equation}
	 \sup _{G' \subseteq G} \Ppc\big( \partial B^{\sss G'}_x (r) \neq \emptyset \big) \le \frac{C_2 }{r}. \label{e:critonearm}
\end{equation}
Note that the quantity $\Pp\big(x \lrr y\big)$ is monotone increasing in $p$ and so \eqref{e:critvolball} holds for any $p \leq p_c$. Furthermore, even though monotonicity in $p$ is unknown to hold for the quantity $\Pp\big(x \lrre y\big)$, the triangle condition \eqref{e:strongtriangle} from which \eqref{e:critonearm} follows \emph{is} monotone in $p$ and therefore \eqref{e:critonearm} holds for any $p \leq p_c$ as well.

Since $|B(r)| = \sum_{j=0}^r |\partial B(j)|$, it is reasonable to expect given \eqref{e:critvolball} that the sequence $\E_{p_c} [|\partial B(r)|]$ is bounded. It is, however, an open problem to show that the triangle condition implies this. In \cite[Theorem 4.1]{HofNac12} it is proved under the stronger conditions \eqref{e:assump1} and \eqref{e:assump2}. In fact, a stronger statement is proved under these assumptions: there exists a constant $C_b>0$ such that for any $G' \subseteq G$ and any $r$ we have
\begin{equation}\label{e:critlayer} \E_{p_c}\big[\big | \partial B^{\sss G'}_x(r) \big | \big]\leq C_b\, .\end{equation}
We remark here that since this estimate relies on conditions \eqref{e:assump1} and \eqref{e:assump2}, it will \emph{not} be used to prove Theorem \ref{thm:general1}.

While \cite{KozNac09} gives a corresponding lower bound for \eqref{e:critvolball} when $G$ is any infinite transitive graph, we obviously cannot expect such a lower bound to be valid for all $r$ when $G$ is a finite graph. In \cite[Lemmas 4.2 and 4.3]{HofNac12} it is proved for any transitive graph that there exist constants $c,\xi>0$ (that may depend on $\lambda$ in \eqref{e:chidef}) such that
\begin{equation} \label{e:critvolballlower}
	\Epc[|B(r)|] \ge \frac{r}{4} \quad \text{ and } \quad \Epc[|\dbr|] \ge c \qquad  \text{ for all } \quad r < \xi V^{1/3}.
\end{equation}
From here it is easy to obtain similar lower bounds for \eqref{e:critonearm} and \eqref{e:critlayer} and this is the content of Lemma~\ref{lem:critonearmlower} below.

Lastly, in \cite{NacPer08} general estimates that bound the probability that a cluster has small volume but large diameter are given. We recall these estimates now. Assume that $G=(\Vcal, \Ecal)$ is a graph with $V$ vertices and $p\in[0,1]$ is such that \eqref{e:critvolball} and \eqref{e:critonearm} hold at $p$ and that $r$ and $k$ are integers satisfying
\[
	 r \geq 16 C_2 k V^{-1/3} \, ,
\]
where $C_2$ is the constant from \eqref{e:critonearm}. Then, by \cite[Lemma 6.2]{NacPer08}, for any $x \in \Vcal$,
\begin{equation}\label{e:smallCdiam}
 \Pp \big ( |\Ccal(x)| \leq k \text{ and } \partial B^{\sss G}_x (r) \neq \emptyset \big ) \leq C_2 \max\big\{2 r^{-1} ,V^{-1/3} \big\}\, 2^{-{r^2 \over (64 C_2+2)k}} \, .
\end{equation}
Furthermore, by \cite[Lemma 6.3]{NacPer08}, if $k$ and $r$ satisfy
\[
	 r \geq 32 C_2 k V^{-1/3} \quad \text{ and } \quad r \geq \sqrt{2(64C_2 + 2)k}\, ,
\]
then,
\begin{equation}\label{e:smallCdiam2}
\Pp \big ( \exists x \in \Vcal \,: \, |\Ccal(x)| \leq k \text{ and } \partial B^{\sss G}_x (r) \neq \emptyset \big ) \leq 4C_2 \max \big\{ 2r^{-1}, V^{-1/3} \big \}\, {V \over k} 2^{-{r^2 \over (64 C_2+2)k}} \, .
\end{equation}
\medskip

The following lemma provides a corresponding lower bound to \eqref{e:critonearm}.
%%%%%%%%%%%%%%%%%%% LEMMA %%%%%%%%%%%%%%%%%%%%%%%%%%%%%%%%%%%%%%%%%%%%%
\begin{lemma}\label{lem:critonearmlower}
Let $(G_m)$ be a sequence of finite transitive graphs satisfying the triangle condition \eqref{e:strongtriangle}.
Then there exist constants $c_3, \zeta > 0$ such that,
\[
	\Ppc(\dbr \neq \emptyset) \ge \frac{c_3}{r} \quad \text{for all } r < \zeta V^{1/3} \, .
\]
\end{lemma}
%%%%%%%%%%%%%%%%%%%%%%%%%%%%%%%%%%%%%%%%%%%%%%%%%%%%%%%%%%%%%%%%%%%%%
\proof We follow the proof of \cite[Theorem 1.3(i)]{KozNac09}, where the equivalent statement is proved for critical percolation on $\Zd$ with $d$ large. For any $a \ge 1$,
\[
	\Ppc(\dbr \neq \emptyset) = \Ppc \big(|B(a r) \setminus B(r-1)| > 0 \big) \ge \frac{\Epc[|B(a r) \setminus B (r-1)| ]^2}{\Epc[|B (a r) \setminus B (r-1)|^2]} \, ,
\]
by the inequality $\P(V>0)\geq \E [V]^2/\E[V^2]$ valid for any non-negative random variable. Let $\xi>0$ be the constant from \eqref{e:critvolballlower}, so that \eqref{e:critvolball} and \eqref{e:critvolballlower} yield that
\[
	 \Epc[|B (a r) \setminus B (r-1)| ]\geq (a/4-C_1)r \quad \text{ for all } r \leq \xi a^{-1} V^{1/3} \, .
\]
Next, by a standard application of the BK-inequality \cite{BerKes85} (see e.g.\ \cite[page 652]{KozNac09} and also footnotes \ref{fn:disj} and \ref{fn:BK} below) we have that $\E_p[|B(r)|^2] \le \E_p[|B(r)|]^3$ and so by \eqref{e:critvolball} we get that $\Epc [|B (ar)|^2] \leq (C_1 a)^3 r^3$. Putting these together gives
\[
	\Ppc(\dbr \neq \emptyset)  \ge \frac{\left(\frac{a}{4} - C_1\right)^2}{C_1^3 a^3 r} \qquad \text{ for all } \quad r \leq a^{-1} \xi V^{1/3}.
\]
We maximize the right-hand side by putting $a = 12C_1$ and choose $\zeta = a^{-1} \xi$ and $c_3$ as the constant we get on the right-hand side above, concluding the proof. \qed
\medskip

We may now use our previous estimates on critical percolation to deduce a simple lower bound on the probability of the one-arm event in the subcritical phase.
%%%%%%%%%%%%%%%%%%% LEMMA %%%%%%%%%%%%%%%%%%%%%%%%%%%%%%%
\begin{lemma}\label{lem:subdiam}
Assume the setting of Theorem~\ref{thm:general1}.
Let $c_3$ and $\zeta$ be the same constants as in Lemma \ref{lem:critonearmlower}. Then the following assertions hold:
\begin{enumerate}
	\item[(a)] There exists a $c>0$ such that for all integers $r < \zeta V^{1/3}$,
\[
	\Ep[|\dbr|] \ge c (1-\vep)^r,
\]
and
 \[
 	\Pp(\dbr \neq \emptyset) \ge \frac{c_3}{r} (1-\vep)^r.
\]

\item[(b)] There exists a constant $c_4>0$ such that for any $A >1$
\[
	 \Pp(\partial B (A \vep^{-1}) \neq \emptyset) \le \vep \e^{-c_4 A}.
\]
\end{enumerate}
\end{lemma}
%%%%%%%%%%%%%%%%%%%%%%%%%%%%%%%%%%%%%%%%%%%%%%%%%%%%%%
\begin{rk}\emph{ Note that the upper bound in part (b) is weaker than the upper bound in Theorem~\ref{thm:generalmaxdiam}(b), but that the assumptions here are also weaker.
}
\end{rk}
\proof
(a) It is an easy consequence (see \cite[Lemma 3.4]{HofNac12} for a proof) of the standard simultaneous coupling between percolation with parameters $p_1$ and $p_2$ satisfying $0 \le p_1 \le p_2 \le 1$ that
\[
	\E_{p_1} [ |\dbr|] \ge \left(\frac{p_1}{p_2}\right)^r \E_{p_2}[|\dbr|],
\]
and
\[
	\P_{p_1} (\dbr \neq \emptyset) \ge \left(\frac{p_1}{p_2}\right)^r \P_{p_2}(\dbr \neq \emptyset).
\]
for any integer $r$. Thus the proof of part (a) is concluded by taking $p_1 = p_c(1-\vep)$ and $p_2 = p_c$ and applying \eqref{e:critvolballlower} and Lemma \ref{lem:critonearmlower}, respectively.
\medskip

(b) Put $r=A\vep^{-1}$ and $k=A \vep^{-2}$ and bound
\[
	 \Pp(\partial B (r) \neq \emptyset) \le \Pp \big( |\Ccal| \le k \text{ and }\partial B (r) \neq \emptyset \big) + \Pp(|\Ccal| \ge k) \, .
\]
Note that as long as $A > 1$ and $V$ is large enough, by \eqref{e:chibds} we have that $k \geq \chi(p)^2$ and since $\vep^3 V \to \infty$ we also have that $r \gg k V^{-1/3}$. Hence we may apply \eqref{e:uppertail} and \eqref{e:smallCdiam} in the above inequality to obtain
\[\begin{split}
	\Pp(\partial B (A \vep^{-1}) \neq \emptyset) & \le \frac{2C_2 \vep}{A} 2^{-\frac{A}{64 C_2+ 2}} + \sqrt{\frac{\e}{A}} \,\vep\, \e^{-A/2} \le \vep \e^{-c_4 A},
\end{split}
\]
for some $c_4>0$ sufficiently small.\qed

\proof[Proof of the lower bounds in Theorems \ref{thm:Qmonearm} and \ref{thm:generalmaxdiam}(b)] These follow immediately from Lemma~\ref{lem:subdiam} and the fact that when $\vep^3 V \to \infty$ we have $\vep^{-1} \log(\vep^3 V) \ll V^{1/3}$. \qed

\section{Cluster sizes: proofs of Theorems \ref{thm:numberlarge} and \ref{thm:lowertail}}\label{sec:mainproof}

We start with bounds on the first and second moment of the typical cluster size, conditioned on the event that the diameter of the cluster is large.

%%%%%%%%%%%%%%%%%%%% LEMMA %%%%%%%%%%%%%%%%%%%%%%%%%%%%%%%%%%%%%%%%%%%%%

\begin{lemma}\label{lem:firstmom}
Assume the setting of Theorem \ref{thm:general1} and let $\zeta>0$ be the constant from Lemma~\ref{lem:subdiam}. Then
there exists a constant $c>0$ such that for any $A \in [1, \zeta \vep V^{1/3}]$
\[
	\Ep \big [|\Ccal| \, \big | \,  \partial B (A \vep^{-1}) \neq \emptyset \big ] \ge c A \vep^{-2} \, .
\]
\end{lemma}
%%%%%%%%%%%%%%%%%%%%%%%%%%%%%%%%%%%%%%%%%%%%%%%%%%%%%%%%%%%%%%%%%%%%%
\proof
We put $r=A\vep^{-1}$ and $k = \alpha A \vep^{-2}$ for some small $\alpha >0$ that will be chosen later. We bound
\[
	\Ep\big [ |\Ccal| \, \big | \, \dbr \neq \emptyset \big ] \,\ge\, k \,\Pp\left(|\Ccal| \ge k \,\mid\, \dbr \neq \emptyset\right) \,=\,  \frac{k \cdot \Pp\left(|\Ccal| \ge k \text{ and } \dbr \neq \emptyset\right)}{\Pp(\dbr \neq \emptyset)}.
\]
Since $\vep^3 V \to \infty$, the conditions of \eqref{e:smallCdiam} hold, so we obtain
\[
	\Pp(|\Ccal| \le k \text{ and }\dbr \neq \emptyset) \le {C \vep \over A} e^{-{A \over C \alpha}} \, ,
\]
for some constant $C>0$. Since $r \leq \zeta V^{1/3}$, by Lemma \ref{lem:subdiam}(a) we get $\Pp(\dbr \neq \emptyset) \geq c\vep A^{-1} e^{-cA}$, so when $\alpha >0$ is chosen to be a small enough (but fixed) we get that
\[
	\frac{\Pp\left(|\Ccal| \ge k \text{ and } \dbr \neq \emptyset\right)}{\Pp(\dbr \neq \emptyset)} \ge 1 - \frac{C \vep A^{-1} \e^{-\frac{A}{C \alpha}}}{c \vep A^{-1} \e^{-c A}} \geq 1/2 \, ,
\]
giving the lemma. \qed \medskip

%%%%%%%%%%%%%%%%% LEMMA %%%%%%%%%%%%%%%%%%%%%%%%%%%%%%%%
\begin{lemma}\label{lem:secmom} Consider percolation on a transitive graph $G$ with parameter $p \in [0,1]$. For any integer $r \ge 1$,
\[
	\Ep\big [|\Ccal|^2 \, \big | \,  \dbr \neq \emptyset \big ] \le \chi(p)^2 r(r+1) + \chi(p)^3 (r+1).
\]
\end{lemma}
%%%%%%%%%%%%%%%%%%%%%%%%%%%%%%%%%%%%%%%%%%%%%%%%%%%%%%%
\proof
For a simple path $\eta$ of length $r$ starting at $v$ we write the event
\begin{equation}\label{e:Agammadef}
	\Acal_r(v,\eta) := \big\{\eta \text{ is the first shortest open path of length $r$} \big\} \, ,
\end{equation}
where by ``first'' we mean according to some fixed predetermined ordering of paths (such as the lexicographical order). In other words, $\Acal_r(v,\eta)$ is the event that $\eta$ is the first open path of length $r$ such that the last vertex of $\eta$ is in $\partial B_v(r)$. Observe that
\begin{equation} \label{gammar}
	{\biguplus_{\eta}} \Acal_r(v, \eta) = \big \{\partial B_v (r) \neq \emptyset \big \}.
\end{equation}
Since the events $\{\Acal_r(v,\eta)\}_\eta$ are mutually disjoint we can write
\[\begin{split}
	\Ep[|\Ccal(v)|^2\indi_{\{\partial B_v (r) \neq \emptyset\}}]
	 = \sum_{x,y} \sum_{\eta} \Pp\left(v \conn x \text{ and } v \conn y \text{ and } \Acal_r(v,\eta) \right).
\end{split}
\]
If the event $\{v \conn x\} \cap \{v \conn y\} \cap \Acal_r(v,\eta) $ occurs, then one of the following events must occur (see Figure \ref{fig:Pf32}):
\begin{enumerate}
\item There exists integers $m\neq n$ with $1 \leq m,n \leq r$ such that the events $\Acal_r(v,\eta)$, $\eta(m) \conn x$ and $\eta(n) \conn y$ occur disjointly,\footnote{\label{fn:disj} Given two events $A$ and $B$, we write $A \circ B$, and say that $A$ and $B$ \emph{occur disjointly} if, given a percolation configuration $\omega$, there exists a set of edges $W_A(\omega)$ so that we can verify whether $\omega \in A$ by examining the status of only edges in $W_A(\omega)$, while we can verify whether $\omega \in B$ by examining the status of only edges in $W_B(\omega) \subseteq \omega \setminus W_{A}(\omega)$. We call $W_A(\omega)$ the set of \emph{witness edges} for $A$.} or,
\item There exists $1 \leq m \leq r$ and a vertex $z$ such that the events $\Acal_r(v,\eta)$, $\eta(m) \conn~z$, $z \conn x$ and $z \conn y$ occur disjointly.
\end{enumerate}

\begin{figure}
	\includegraphics[width = 0.8\textwidth]{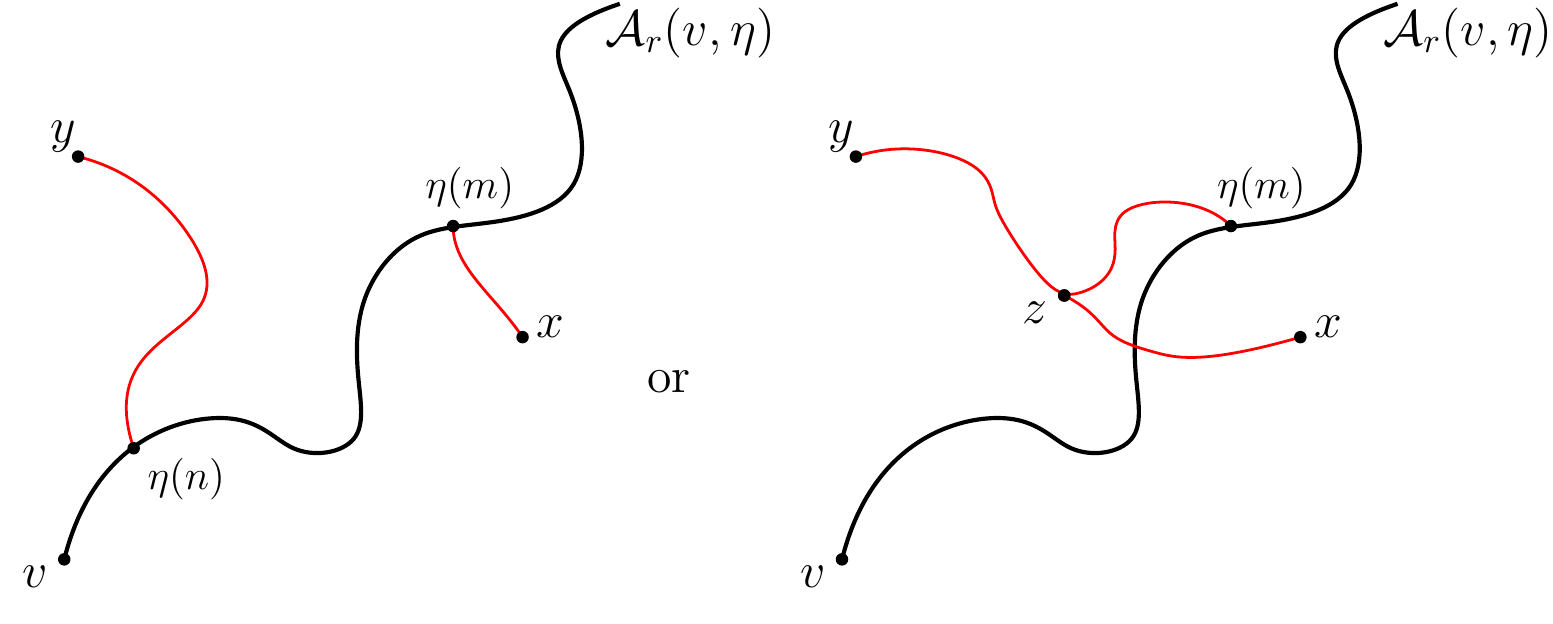}
	\caption{\label{fig:Pf32} The two cases of $\{v \conn x\} \cap \{v \conn y\} \cap \Acal_r(v,\eta)$.}
\end{figure}

To see this implication consider an open path from $x$ to $v$ and let $\gamma_x$ be the part of this path from $x$ until the first time it hits $\eta$, so that $\gamma_x$ and $\eta$ are edgewise disjoint. (If $x$ is a vertex on $\eta$ then $\gamma_x=\emptyset$). Now consider another open path, from $y$ from $v$ and let $\gamma_y$ be the part of this path from $y$ until the first time it hits $\eta \cup \gamma_x$. If $\gamma_y$ ends at $\eta$ rather than $\gamma_x$, then this is an instance of case (i) above when we write $m,n$ for the positions on $\eta$ of the meeting points of $\gamma_x$ and $\gamma_y$ with $\eta$, respectively.  If $\gamma_y$ ends at $\gamma_x$ instead of $\eta$, then this is an instance of case (ii) above when we write $z$ for that meeting point and $m$ for the position on $\eta$ of the meeting point of $\eta$ with $\gamma_x$.

In case (i) the disjoint witnesses for the occurrence of the events are the edges of $\eta$ together with all the closed edges (these open and closed edges determine $\Acal_r(v,\eta)$ since one can check that $\eta$ is open and any other path of length $r$ that is prior to $\eta$ in the fixed ordering has a closed edge in it), the edges of $\gamma_x$ (for $\eta(m) \conn x$) and the edges of $\gamma_y$ (for $\eta(n) \conn y$). Similarly, in case (ii) the witnesses are the edges of $\eta$ together with all closed edges, the edges on $\gamma_x$ from $\eta(m)$ to $z$, the edges of $\gamma_x$ from $z$ to $x$ and the edges of $\gamma_y$.

BKR-inequality \footnote{\label{fn:BK} The \emph{van den Berg-Kesten-Reimer inequality} (or BKR-inequality) states that disjoint events are negatively correlated, i.e., $\Pp(A \circ B) \le \Pp(A) \Pp(B)$. If $A$ and $B$ are increasing events (i.e., if $\Pp(A) \le \P_{q}(A)$ for all $0 \le p \le q \le 1$), then we call this bound the BK-inequality \cite{BerKes85}. The BK-inequality is usually easier to apply, because it is easy to verify whether increasing events occur disjointly. Applying the BKR-inequality to non-increasing events (such as $\{\partial B(r) \neq \emptyset\}$) often requires more care, see \cite[Section~3]{HofNac12} for a discussion.} \cite{Reim00} now yields

\[\begin{split}
	\Ep[|\Ccal|^2\indi_{\{\dbr \neq \emptyset\}}] &\le  \sum_{x,y} \sum_{m\neq n}^r  \sum_{\eta}  \Pp(\Acal_r(v,\eta)) \Pp(\eta(m) \conn x) \Pp(\eta(n) \conn y) \\ & \quad + \sum_{x,y,z} \sum_{m=0}^r \sum_{\eta}  \Pp(\Acal_r(v,\eta)) \Pp(\eta(m) \conn z) \Pp(z \conn x) \Pp(z \conn y) .
\end{split}
\]
For the term on the first right-hand side we first sum over $x$ and $y$ and get $\chi(p)^2$, then over $m,n$ and get a $r(r+1)$ factor and lastly the sum over $\eta$ gives another factor $\P_p (\dbr \neq \emptyset)$. For the second term we first sum over $x$ and $y$, then over $z,m,\eta$ and get $\chi(p)^3(r+1) \P_p (\dbr \neq \emptyset)$, concluding the proof of the lemma. \qed
\medskip

\proof[Proof of Theorem \ref{thm:general1}] We start with the proof of part (b) of the theorem, which is a straightforward application of the two previous lemmas and a second moment bound. Let $c,\zeta>0$ be the constants from Lemmas \ref{lem:subdiam} and \ref{lem:firstmom}. Put $A \in [2/c, \zeta \vep V^{1/3}]$ and $r=A\vep^{-1}$ and $k = (c/2)A\vep^{-2}$. Recall that for any non-negative random variable $V$ we have $\P( V \geq a ) \geq (\E[ V]-a)^2/\E[V^2]$ for any $a \leq \E [V]$. Hence, by Lemma \ref{lem:firstmom} we may bound
\begin{equation}\label{e:secmom}
\begin{split}
	\Pp(|\Ccal| \ge k) &\,\ge\,  \Pp(\dbr \neq \emptyset) \Pp(|\Ccal| \ge k \,\mid\, \dbr \neq \emptyset) \\
	& \,\ge\, \Pp (\dbr \neq \emptyset) \frac{\left(\Ep\left[|\Ccal| \,\mid\, \dbr \neq \emptyset \right] -k\right)^2}{\Ep\left[|\Ccal|^2 \,\mid\, \dbr \neq \emptyset\right]}.
\end{split}
\end{equation}
Now we apply the bounds from Lemmas \ref{lem:subdiam}, \ref{lem:firstmom}, and \ref{lem:secmom}, and \eqref{e:chibds} and get
\begin{equation}\label{e:clustertail}
\begin{split}
	\Pp(|\Ccal| \ge k ) &\ge  \frac{c_3 \vep}{A}(1-\vep)^{A \vep^{-1}} \cdot \frac{\left(c/2 A \vep^{-2} \right)^2}{C_4 A^2 \vep^{-4}}
	\ge  \frac{c_2 \vep}{A} \e^{-A},
\end{split}
\end{equation}
concluding the proof of part (b).
\medskip

To prove part (a), let $\alpha\in(0,1)$ be arbitrary, and let $\delta=(1-\alpha)$. By part (b) of this theorem we may choose some $c>0$ so that when $k=c \vep^{-2} \log(\vep^3 V)$ we have
\[
	\Pp(|\Ccal| \ge k) = \Omega(\vep (\vep^3 V)^{-\delta}).
\]
Write $\Zk$ for the random variable counting the number of vertices in clusters of size at least $k$, i.e.,
\begin{equation}\label{e:Zkdef}
	\Zk := \# \big \{x \, : \, |\Ccal(x)| \ge k \big \} \, ,
\end{equation}
so that $\E_p[\Zk] = \Omega ( \vep V (\vep^3 V)^{-\delta})$. By the pigeonhole principle we have that for $s \ge t$,
\[
	\big\{\Zk \ge s k \big\} \cap \big\{|\Ccal_1| \le t k\big\} \subseteq \big\{|\Ccal_{\lfloor s / t \rfloor}| \ge k\big\}.
\]
We now let $j$ be an integer satisfying $j \in [1, (\vep^3 V)^\alpha]$ and put $t = 3/c$ and $s=jt$. It follows from \eqref{e:largeclusterupper} that $\Pp(|\Ccal_1| \ge tk) = o(1)$. Hence, it remains to show that $\Pp(\Zk \ge s k) = 1-o(1)$. By the Paley-Zygmund inequality,
\[
	\Pp(\Zk \ge s k) \ge \frac{\left( \Ep[\Zk] - s k\right)^2}{\Ep[\Zk]^2 + \Var(\Zk)} \, ,
\]
when $sk < \Ep[\Zk]$.
Since $j \leq (\vep^3 V)^\alpha$ and $\vep^3 V \to \infty$ and $\delta<1$ we have that $sk = o(\Ep[\Zk])$. Lastly, it is shown in \cite[Lemma 7.1]{BorChaHofSlaSpe05a} that $\mathrm{Var}_p(\Zk) \le V \chi(p)$ and so by \eqref{e:chibds} we obtain that $\Var(\Zk) = o(\Ep[\Zk]^2)$, concluding the proof.\qed

\section{Improved bounds on the one-arm probability}\label{sec:onearmsharp}

In this section we prove upper bounds (that give the sharp exponents) for the probability of the one-arm event (improving upon Lemma \ref{lem:subdiam}(b)) and on the expected size of the boundary, thus completing the proof of Theorem~\ref{thm:generalmaxdiam}(b). In the next section we will use these to prove the upper bound in Theorem~\ref{thm:Qmonearm} and to prove Theorem~\ref{thm:generalmaxdiam}(a).

\subsection{The off-method and bounds on the probability of a long connection}\label{sec:offmethod}
For the proofs we will require two useful estimates from \cite{HofNac12}. The first is a sharp upper bound on the connection probabilities between any two vertices by an open path that is longer than $\mnot$, see \cite[Section 3.4]{HofNac12} and in particular Lemma~3.15 of that paper for the proofs. We do not quote the precise statements from \cite{HofNac12}, but rather state only the consequences that we require in this paper.

One of these bounds, and several more below, make use of the so-called \emph{off-method}.
Given a graph $G = (\Ecal, \Vcal)$ and a subset of the \emph{edge set} $A \subset \Ecal$, we say that an event ``$F$ occurs off $A$'' if $F$ occurs without using any edges in $A$.\footnote{In the literature, the off-method is usually applied with reference to a vertex set, implicitly using the set of all edges that contain a vertex of $A$ in the graph. Here we use an edge set because the traditional definition is a bit unwieldy in our setting. All results from the literature that we use are valid with our more general definition.} More precisely, given a configuration $\omega \in \{0,1\}^{\Ecal}$, let $\omega_A$ be the configuration such that $\omega_A(e) = \omega(e)$ if $e \notin A$, and $\omega_A(e) = 0$ if $e \in A$. Then $\omega \in \{F$ off $A\}$ iff $\omega_A \in F$.
We frequently write $\Pp^A$ for the measure $\Pp^{A}(E) = \Pp(E \off A)$, and similarly, we write $\E_p^A$. Note that $\Pp^A$ is a product measure on $\{0,1\}^{\Ecal \setminus A}$.
 We use the off-method to factorize probabilities. The off-method, for example, can be used to enforce independence, since
\[
	\Pp(\{E \text{ off  }A\} \cap \{F \text{ off }A^c\}) =  \Pp(E \text{ off  }A) \Pp (F \text{ off }A^c).
\]
We allow $A =A(\omega)$, that is, the set $A$ may depend on the configuration. In particular, we will often take $A$ to be a metric ball, i.e., we consider events of the form $\{E \off B_x(r)\}$. In this case we take $A = A(\omega)$ to be
the set of all open edges on a path of open edges of length at most $r$ started at $x$, and of all closed edges that share an end-point with one or two of those open edges.
Observe that we can indeed determine what $B_x(r)$ is for any given $\omega$ by inspecting only the status of the edges in $A(\omega)$. In this setting, $\P_p^{A(\omega)}$ is of course no longer a product measure. We deal with this difficulty whenever it occurs below by using an appropriate conditioning scheme.

Recall the definition of the non-backtracking walk mixing time $\tmix$ defined in \eqref{e:mnotdef} above.

\begin{lemma}[Uniform connection bounds, \cite{HofNac12}]\label{lem:unifconn} Let $G = (\Vcal,\Ecal)$ be a transitive graph satisfying \eqref{e:assump1} and \eqref{e:assump2} and $p \leq p_c$. Then for any vertices $x$ and $y$,
\begin{equation}\label{e:unifconnprob2}
\Pp \big(x \stackrel{\geq \mnot}{\lrfill} y\big) \leq 3V^{-1} \chi(p) \, ,
\end{equation}
and for any $t \ge \mnot$ and any $A \subset \Ecal$
\begin{equation}\label{e:unifconnprob}
	\Pp \big(x \stackrel{=t}{\lrfill} y \text{\emph{ off }} A \big) \le 3V^{-1} \Ep[|\partial B (t - \mnot) \text{\emph{ off }} A|].
\end{equation}
\end{lemma}
The heuristics behind the above lemma are that when a graph satisfies \eqref{e:assump1} and \eqref{e:assump2} a long percolation path has similar properties to a simple random walk path.

The second estimate we need from \cite{HofNac12} is a non-backtracking random walk estimate bounding a particular sum of the heat kernel of graphs, like the hypercube, that satisfy \eqref{e:assump2}. Its proof is not difficult and can be found in the last paragraph in the proof of Theorem 4.5 of \cite{HofNac12}.

\begin{lemma}\label{lem:heatkernelbound} Consider the non-backtracking random walk kernel $\p$ on a transitive graph satisfying \eqref{e:assump2}. Then
\[
	 \sum_{v} \sum_{ t \in [2,\mnot], s \in [1,t]} s \p^s(0,v) \p^t(0,v) = O(\alpha_m / \log V) \, .
\]
\end{lemma}

\subsection{The expected volume of the boundary of a subcritical ball}
We prove the volume bound in Theorem~\ref{thm:generalmaxdiam}(b) in a slightly stronger version, allowing the bound to be ``off'' any arbitrary set of vertices.
\begin{theorem}\label{thm:subvolume}
Assume the setting of Theorem \ref{thm:generalmaxdiam}.
 There exists a constant $C>0$ such that for all integers $r =O(\vep^{-1} \log (\vep^3 V))$, we have
\begin{equation}\label{e:subvolind}
	\sup_{A \subset \Ecal} \, \E_{p} [\#\{v : 0 \stackrel{=r}{\lrfill} v \off A\}] \le C(1-\vep)^{r} \, .
 \end{equation}
\end{theorem}
\proof
We prove the claim by induction on $r$.
 The induction hypothesis is that \eqref{e:subvolind} holds for any integer $k < r$.
The induction is initialized by choosing $C$ sufficiently large.

We start by setting up a coupling that allows us to use the BKR-inequality.
Let $G' = (\Vcal, (\Ecal_1, \Ecal_2))$ be the multigraph with a pair of edges $e_1 \in \Ecal_1$ and $e_2 \in \Ecal_2$ between $v,w \in \Vcal$ iff $\{v,w\} \in \Ecal$ (i.e., we take $G$ and replace each edge by a pair of parallel edges).
Put $p_1=p_c(1-\vep)$ and $p_2 = p_c$.
Independently of everything else, we declare each edge in $\Ecal_1$ open with probability $p_1$ and each edge in $\Ecal_2$ to be open with probability $q$, where $q\in[0,1]$ is determined by
\[
 (1-q)(1-p_1) = 1-p_2 \, .
\]
We write $\P^A$ for the associated product measure off $A$ (i.e., all edges in $\Ecal_1$ and $\Ecal_2$ corresponding to some edge in $A$ are closed).
We say that an edge $e \in \Ecal$ is $p_1$-open iff $e_1$ is open, and that $e$ is $p_2$-open iff at least one of $e_1$ or $e_2$ are open.
For $i=1,2$ we write $G_{p_i}$ for the graph spanned by the $p_i$-open edges. Note that the marginal law of $G_{p_i}$ is $\P_{p_i}$.

For an integer $r \geq 0$ we define
\[
    \A_{r,p_2}(v) := \big \{ 0 \stackrel{=r}{\lrfill} v \text{ in }G_{p_2} \big \},
\]
and given a simple path $\eta$ in $G$ from $0$ to $v$ of length $r$ we define
\[
    \A_{r, p_2}(v,\eta) := \big \{ \eta \text{ is the first $p_2$-open shortest path of length $r$ from $0$ to $v$} \big \} \, ,
\]
so that $\A_{r,p_2}(v) = \uplus_{\eta} \A_{r, p_2}(v,\eta)$. We also define $\Bcal_{r,p_1}(v,\eta)$ to be the event that the edges of $\eta$ are $p_1$-open. It follows that
\[
	 \biguplus_\eta \big(\A_{r, p_2}(v,\eta) \cap \Bcal_{r,p_1}(v,\eta)\big) \subseteq \big \{ 0 \stackrel{=r}{\lrfill} v \text{ in }G_{p_1} \big \} \, .
\]
We will show using the induction hypothesis that
\begin{equation} \label{e:subvolumegoal}
	\supa \sum_{v} \P^A \Big( \big\{0 \stackrel{=r}{\lrfill} v \text{ in }G_{p_1} \big\}\setminus \uplus_\eta \big(\A_{r, p_2}(v,\eta) \cap \Bcal_{r,p_1}(v,\eta)\big) \Big) = o((1-\vep)^r) \, .
\end{equation}
This establishes the proof, since then
\[
	 \supa \E_{p}[\# \{ v :  0 \stackrel{=r}{\lrfill} v \off A\}] \leq  \supa \sum_v \sum_\eta  \P^A \big(\A_{r, p_2}(v,\eta) \cap \Bcal_{r,p_1}(v,\eta)\big) + o((1-\vep)^r) \, ,
\]
while
\[
	\P^A \big(\Bcal_{r,p_1}(v,\eta) \mid \A_{r, p_2}(v,\eta)\big) = (1-\vep)^r
\]
whenever $\P^A(\A_{r, p_2}(v,\eta))>0$, so that
\[\begin{split}
	\supa \sum_v \sum_\eta \P^A \big(\A_{r, p_2}(v,\eta) \cap \Bcal_{r,p_1}(v,\eta)\big) &= (1-\vep)^r \supa \sum_v \P^A \big(0 \stackrel{=r}{\lrfill} v \text{ in }G_{p_2} \big)\\
 & \leq (C_b+o(1)) (1-\vep)^r \, ,
\end{split}
\]
where $C_b$ is the constant from \eqref{e:critlayer}.
\medskip

It remains to prove \eqref{e:subvolumegoal}. Fix a set $A \subset \Ecal$. To start, we assume that the event
\begin{equation}\label{e:subvolevent}
	\big\{0 \stackrel{=r}{\lrfill} v \text{ in }G_{p_1} \big\} \setminus \biguplus_{\eta'} \big(\A_{r,p_2}(v,\eta') \cap \Bcal_{r,p_1}(v,\eta')\big)
 \end{equation}
occurs off $A$, and that $\eta$ is the first shortest $p_1$-open path connecting $0$ to $v$. Since $\A_{r, p_2}(v,\eta) \cap \Bcal_{r,p_1}(v,\eta)$ does not occur, we deduce that either
\begin{enumerate}
	\item the shortest $p_2$-open path connecting $0$ to $v$ has length less than $r$, or
	\item that both are of length at least $r$ but \emph{the first} shortest $p_2$-open path uses an edge that belongs to $\Ecal_2$.
\end{enumerate}
Both cases imply that there are vertices $x$ and $y$ on $\eta$ such that the length of $\eta$ between them is some $t\leq r$ and there exists a $p_2$-open path $\gamma$ between them with $\gamma \cap \eta = \{x,y\}$ and $|\gamma| \leq t$, and $\gamma$ contains at least one edge of $\Ecal_2$. See Figure \ref{fig:Pf43}.
\begin{figure}
	\includegraphics[width = 0.75\textwidth]{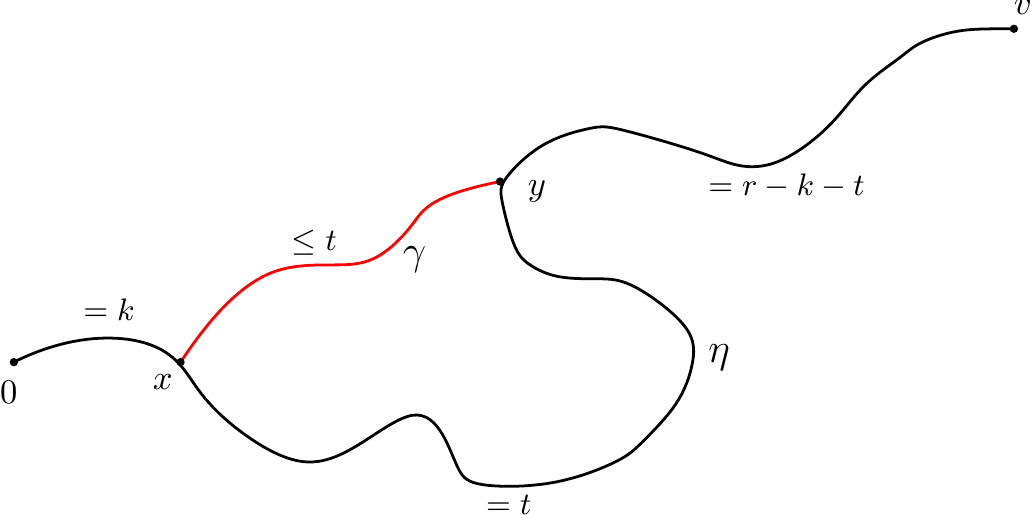}
	\caption{\label{fig:Pf43} The event implied by $\big\{0 \stackrel{=r}{\lrfill} v \text{ in }G_{p_1} \big\} \setminus \biguplus_{\eta'} \big(\A_{r,p_2}(v,\eta') \cap \Bcal_{r,p_1}(v,\eta')\big)$. Here the black path is $\eta$ and the red path is $\gamma$, which passes through at least one edge of $\Ecal_2$.}
\end{figure}
Hence, the event \eqref{e:subvolevent} implies that there exists non-negative integers $k,t$ satisfying $k+t \leq r$ and vertices $x,y$ such that the following two events occur disjointly:
\[\begin{split}
	\Lcal_1(v,x,y,k,t) &:= \Big\{ \big\{0 \stackrel{=k}{\lrfill} x  \big\} \cap \big\{x \stackrel{=t}{\lrfill} y \off B_0(k) \big\} \cap \big\{y \stackrel{=r-k-t}{\lrfill} v \off  B_{0}(k+t)\big\} \text{ in } G_{p_1} \Big\},\\
\Lcal_2(x,y,t) & := \big\{\exists \gamma \, : \, \gamma \text{ is  a } p_2\text{-open path, $|\gamma| \le t$,  $\gamma(0) =x$,  $\gamma(|\gamma|) =y$, } \gamma \cap \Ecal_2 \neq \emptyset  \big\}.
\end{split}
\]
Indeed, the witness edges for $\Lcal_1$ are the $p_1$-open edges of $\eta$ together with all the closed edges of $\Ecal_1$, and the witness edges for $\Lcal_2$ are the open edges of $\gamma$. Denote the event of the disjoint occurrence of $\Lcal_1$ and $\Lcal_2$ by $\F(v,x,y,k,t)$. We will prove \eqref{e:subvolumegoal} by summing the probability of $\F(v,x,y,k,t)$ over $v,x,y,k,$ and~$t$.
\medskip

We split the sum according to whether $t < \mnot$ or $t\geq \mnot$, starting with the latter.
\medskip

Applying the BKR-inequality and using the inclusion $\Lcal_2(x,y,t) \subseteq \{x \stackrel{\le t}{\lrfill} y$ in $G_{p_2}\}$ we bound
\[
	\P^A (\F(v,x,y,k,t))  \leq \P_{p_c}\big(x \stackrel{\le t}{\lrfill} y\big) \P^A\big(\Lcal_1(v,x,y,k,t)\big).
\]
(We dropped the condition ``off $A$'' for the first factor because the event is increasing.)

We proceed by bounding $\P^A(\Lcal_1(v,x,y,k,t))$. We condition on the open and closed edges that determine $B_0(k+t)$, as described in Section \ref{sec:offmethod}, and use the induction hypothesis to get
\[
	\sum_v  \P^A (\Lcal_1(v,x,y,k,t)) \le C (1-\vep)^{r-k-t} \Pp^A \big(0 \stackrel{=k}{\lrfill} x , x \stackrel{=t}{\lrfill} y \text{ off } B_0(k) \big).
\]

We condition similarly on the closed and open edges that determine $B_0(k)$, and since $t \geq \mnot$ and we assume that \eqref{e:assump1} and \eqref{e:assump2} hold, we may use \eqref{e:unifconnprob} and the induction hypothesis to bound
\[
	 \sum_v  \P^A (\Lcal_1(v,x,y,k,t)) \leq C V^{-1}(1-\vep)^{r-k} \P_{p}^A \big(0 \stackrel{=k}{\lrfill} x  \big) \, .
\]
Thus,
\[
	 \sum_v \P^A(\F(v,x,y,k,t)) \leq C V^{-1}(1-\vep)^{r-k} \P_{p}^A \big(0 \stackrel{=k}{\lrfill} x  \big)\P_{p_c}\big(x \stackrel{\le t}{\lrfill} y \big) \, .
\]
We now sum the last term over $y$ and get a factor $O(t)$ by \eqref{e:critvolball}.
We then sum the one before last term over $x$ and get a factor $C(1-\vep)^k$ by the induction hypothesis. Finally, we sum over $k,t \leq r$ and get a factor $O(r^3)$, obtaining
\[
	 \sum_{v,x,y} \sum_{k,t \geq \mnot} \P^A(\F(v,x,y,k,t))\leq C (1-\vep)^r V^{-1} r^3.
\]
Since $r = o(V^{1/3})$, we get that this sum is $o((1-\vep)^r)$ for any fixed $A \subset \Ecal$, as required.
\medskip

We now bound in the case that $t \in [2,\mnot]$. Again we start by applying the BKR-inequality to the probability of $\F(v,x,y,k,t)$. This time we bound the probability of $\Lcal_2$ by enumerating over paths. Indeed, $\Lcal_2(x,y,t)$ implies that there exists a path $\gamma$ with $|\gamma|\leq t$ such that $\gamma$ is a $p_2$-open path between $x$ and $y$ such that one of its edges belongs to $\Ecal_2$. For each such simple path $\gamma$ of length $s \leq t$ the probability that this occurs is precisely
\[
	 p_2^{s} (1-(p_1/p_2)^{s}) \, ,
\]
and the number of such $\gamma$'s is at most $m(m-1)^{s-1}\p^s(x,y)$. Hence
\[
	 \Ppc(\Lcal_2(x,y,t)) \leq \sum_{s \leq t} p_c^{s} (1-(1-\vep)^{s}) m(m-1)^{s-1}\p^s(x,y)
	 \leq C \vep \sum_{s \leq t} s \p^s(x,y) \, ,
\]
where in the last inequality we used that $(1-(1-\vep)^s) \leq s\vep$ and that $p_c^s m(m-1)^{s-1}=1+o(1)$ by \eqref{e:assump1}.

For the probability of $\Lcal_1$, we first sum over $v$ as before to get a factor $(1-\vep)^{r-k-t}$. Afterwards, we condition on the closed and open edges that determine $B_0(k+t)$ and bound the conditional probability of $x \stackrel{=t}{\lrfill} y$ by $C(1-\vep)^t\p^t(x,y)$ as before, by enumerating paths and using \eqref{e:assump1}. We gained the factor $(1-\vep)^t$ relative to the estimate of $\Ppc(\Lcal_2(x,y,t))$, because the event $x \stackrel{=t}{\lrfill} y$ occurs on $G_{p_1}$, where the percolation probability is $p_1 = p_c(1-\vep)$.
 We get that
\[
	 \sum_{v,x,y} \sum_{k, t \leq \mnot} \P^A (\F(v,x,y,k,t))
	 \leq C \vep \sum_{x,y} \sum_{k,t \leq \mnot, s\leq t} (1-\vep)^{r-k} \P_{p}^A \big(0 \stackrel{=k}{\lrfill} x \big) s \p^s(x,y)\p^t(x,y) \, .
\]
By Lemma \ref{lem:heatkernelbound} we may sum $s \p^s(x,y)\p^t(x,y)$ over $s,t,y$ and get a factor $O(\alpha_m/\log V)$. We then sum over $x$ using the induction hypothesis to get a factor $C(1-\vep)^k$. Finally we sum over $k$ and get a factor $r$. This yields
\[
	 \sum_{v,x,y} \sum_{k, t \leq \mnot} \P^A (\F(v,x,y,k,t))\leq C (1-\vep)^r \frac{\vep r \alpha_m}{ \log V} \, .
\]
Now, since $r = O(\vep^{-1} \log (\vep^3 V))$ and $\alpha_m = o(1)$ we get that this is also $o((1-\vep)^r)$ for any fixed $A$, as required. \qed

\subsection{The subcritical one-arm probability}
The next theorem gives the sharp estimate on the subcritical one-arm probability in Theorem~\ref{thm:generalmaxdiam}(b) (again, in the slightly stronger form allowing it to be ``off'' any arbitrary set). The proof is of similar nature to the proof of the previous theorem but is not quite analogous, because here the case $t \geq \mnot$ gives rise to a technical difficulty when $\vep$ is very close to $V^{-1/3}$.

Note also that Theorem~\ref{thm:generalmaxdiam}(b) is not entirely sharp, as it does not meet the lower bound of Lemma \ref{lem:subdiam}(a). However, we only use this theorem with $r$ of order $\vep^{-1} \log(\vep^3 V)$, so the ratio between the lower and the upper bound is at most $\log(\vep^3 V)$ and this logarithmic difference should, in practice, not matter much. Our bounds can be improved to give
the sharpest upper bound of order $r^{-1} (1-\vep)^r$, but this seems to require longer technical work and is unnecessary for our purposes, so we omit it. (The current proof actually gives an upper bound of $\vep (1-\vep)^r / \log m$, but we also do not spell out the details for this.)

\begin{theorem}\label{thm:subdiam}
Assume the setting of Theorem \ref{thm:generalmaxdiam}.
There exists a constant $C>0$ such that for all integers $r$ satisfying $\vep^{-1} \le r =O(\vep^{-1} \log (\vep^3 V))$, we have
\begin{equation}\label{e:subonearm}
 	\sup_{A \subset \Ecal} \P_{p}(\partial B(r) \neq \emptyset \off A) \le C\vep(1-\vep)^{r}.
\end{equation}
\end{theorem}

\proof
We again prove the claim by induction. Our induction hypothesis is that \eqref{e:subonearm} holds for any $k$ satisfying  $\vep^{-1} \le k < r$. The induction is initialized by observing that for $r=\vep^{-1}$ the claim follows from \eqref{e:critonearm}.

As in the proof of the previous theorem, we start by constructing the multigraph $G'$ that is a copy of $G$ with each edge replaced with a pair of edges subject to different percolation probabilities, $p_1$ on $\Ecal_1$ and $q$ on $\Ecal_2$, where $q$ is the solution to $(1-q)(1-p_1) =1-p_2$. Also as in the previous proof, we put $p_1=p_c(1-\vep)$ and $p_2 = p_c$. We use the terms ``$p_1$-open'' and ``$p_2$-open'' as before, and write $G_{p_i}$ for the subgraph of $G'$ of $p_i$-open edges.

Define for $i=1,2,$
\[
    \A_{r, p_i} := \big \{ \partial B_0(r) \neq \emptyset \text{ in }G_{p_i}\big \} \, ,
\]
and given a simple path $\eta$ in $G$ of length $r$ we write
\[
    \A_{r, p_i} (\eta) := \big \{ \eta \text{ is the lexicographical first $p_i$-open shortest-path of length $r$ starting at $0$} \big \} \, ,
\]
so that $\A_{r, p_i} = \uplus_{\eta} \A_{r,p_i}(\eta)$. We also write $\Bcal_{r,p_1}(\eta)$ for the event that the edges of the path $\eta$ are $p_1$-open. Note that
\[
	 \biguplus_\eta \big(\A_{r,p_2}(\eta) \cap \Bcal_{r,p_1}(\eta)\big) \subseteq \A_{r,p_1} \, .
\]
We will use the induction hypothesis to show that
\begin{equation} \label{e:onearmgoal}
	\supa \P^A \Big ( \A_{r,p_1} \setminus \uplus_\eta \big(\A_{r,p_2}(\eta) \cap \Bcal_{r,p_1}(\eta)\big) \Big ) = o(\vep (1-\vep)^r) + o\big(\supa \P^A(\A_{r,p_1})\big) \, .
\end{equation}
Given \eqref{e:onearmgoal} the proof can be quickly completed since we have
\[
	\supa \P^A(\A_{r,p_1}) \leq  \supa \sum_\eta \P^A \big(\A_{r,p_2}(\eta) \cap \Bcal_{r,p_1}(\eta) \big) + o(\vep(1-\vep)^r) + o \big(\supa \P^A(\A_{r,p_1})\big)  \, ,
\]
and
\[
	\P^A \big(\Bcal_{r,p_1}(\eta) \mid \A_{r,p_2}(\eta)\big) = (1-\vep)^r
\]
whenever $\P^A(\A_{r,p_2}(\eta))>0$, so that
\[
	 \supa \P^A(\A_{r,p_1}) \leq (1-\vep)^{r}  \supa \P_{p_c}^A(\partial B_r(0) \neq \emptyset) + o(\vep(1-\vep)^r) + o\big(\supa \P^A(\A_{r,p_1})\big) \, ,
\]
which concludes the proof using \eqref{e:critonearm} since $r \geq \vep^{-1}$.
\medskip

We now turn to proving \eqref{e:onearmgoal}. Assume that the event
\begin{equation}\label{e:onearmgoalevent}
	\A_{r,p_1} \setminus \biguplus_{\eta'} \big(\A_{r,p_2}(\eta') \cap \Bcal_{r,p_1}(\eta') \big)
\end{equation}
occurs and let $\eta$ be the first $p_1$-open shortest path of length $r$ starting at $0$. Since $\A_{r,p_2}(\eta) \cap \Bcal_{r,p_1}(\eta)$ does not occur, it follows that either
\begin{enumerate}
	\item there is a $p_2$-open path between the endpoints of $\eta$ of length less than~$r$, or
	\item there exists a $p_2$-open path of length $r$ connecting the two endpoints of $\eta$ that is lexicographically prior to $\eta$.
\end{enumerate}
Fix a set $A \subset \Ecal$.

Let $\Vcal(\eta)$ denote the set of the $r+1$ vertices on the path $\eta$. Both cases (i) and (ii) imply that there exist vertices $u,v \in \Vcal(\eta)$ that are connected by a $p_2$-open path $\gamma$ that is disjoint from $\eta$, and additionally, that this path has at least one edge that is $p_1$-closed but $p_2$-open. We write $\F^{\ge} (\eta)$ for the event that there exists such a $\gamma$ with $|\gamma| \geq \mnot$ and by $\F^{\le}(\eta)$ the event that all such $\gamma$'s have length less than $\mnot$. Our goal is to bound from above the probability of $\uplus_\eta \big( \F^{\ge} (\eta) \cup \F^{\le}(\eta)\big)$ by the right-hand side of \eqref{e:onearmgoal}.
\medskip

\label{page:Fproof1}
We start with $\F^{\ge}(\eta)$, which is simpler to analyse. Here we drop the requirement that one of the edges of $\gamma$ is $p_1$-closed. The event $\F^{\ge}(\eta)$ implies that there exists $x,y \in \Vcal(\eta)$ such that the events
\begin{enumerate}
\item $\eta$ is the first $p_1$-open shortest path of length $r$ starting from $0$, and
\item there exists a $p_2$-open path $\gamma$ from $x$ to $y$ of length at least $\mnot$ that is disjoint from $\eta$,
\end{enumerate}
occur disjointly. Indeed, the witness set for the first event is the set of edges of $\eta$ and all the $p_1$-closed edges (the closed edges determine that $\eta$ is the first shortest $p_1$-open path), and the witness set for the second event is the set of (open) edges of $\gamma$. These witness sets are disjoint by construction. By the BKR-inequality and the union bound we get
\[
	 \P^A \big(\F^{\ge} (\eta) \big) \leq \P^A \big(\A_{r,p_1}(\eta) \big) \sum_{x,y \in \Vcal(\eta)} \P^A \big(x \stackrel{\geq \mnot}{\lrfill} y \big) \, .
\]
By \eqref{e:unifconnprob2} and \eqref{e:chibds},
\[
	 \P^A \big(\F^{\ge} (\eta)\big) \leq \P^A \big(\A_{r,p_1}(\eta)\big) \, |\Vcal(\eta)|^2 \, C V^{-1} \vep^{-1} \leq C \P^A \big(\A_{r,p_1}(\eta) \big) V^{-1} r^2 \vep^{-1} \, ,
\]
and since $V^{-1} r^2 \vep^{-1} = o(1)$ by the assumptions on $r$ and $\vep$, we get that
\[
	 \supa \P^A \big (\uplus_\eta \F^{\ge } (\eta) \big ) \leq  \supa \P_{p}^A(\partial B_r(0) \neq \emptyset) \cdot o(1) \, ,
\]
corresponding to the last term on the right-hand side of \eqref{e:onearmgoal}.
\medskip

\label{page:Fproof}
To bound the probability of $\uplus_\eta \F^{\le}(\eta)$ we observe that this union implies that there exist non-negative integers $k,t, \ell$ with $k+t \leq r,\ell \leq \mnot$ and vertices $x,y$ such that the following two events occur disjointly:
\[\begin{split}
	\Mcal_1(x,y,k,t) & := \big\{\{ 0 \stackrel{=k}{\lrfill} x\} \cap \{x \stackrel{=t}{\lrfill} y \off B_0(k)\}\\
	& \qquad \qquad \qquad \qquad  \cap \{\partial B_y (r-k-t) \neq \emptyset \off B_0(k+t)\} \text{ in } G_{p_1}\big\},\\
	\Mcal_2(x,y,\ell) & := \big\{ \exists \gamma \, : \,  \gamma \text{ is a $p_2$-open path,} |\gamma| =\ell, \gamma(0) = x, \gamma(\ell) =y, \gamma \cap \Ecal_2 \neq \emptyset \big\}.
\end{split}\]
Indeed, as before, the witness set for the first event is $\eta$ and all the $p_1$-closed edges, while the witness set for the second event are the edges of $\gamma$. These witness sets are again disjoint. The BKR-inequality gives
\begin{equation}\label{e.onearmboundevent}
	\P^A \big (\uplus_\eta \F^{\le}(\eta) \big ) \leq \sum_{x,y} \sum_{\substack{k,t,\ell\leq\mnot \\ k+t \leq r}} \P^A (\Mcal_1(x,y,k,t)) \P^A (\Mcal_2(x,y,\ell)) \, .
\end{equation}

To bound the probability of the second event, we enumerate all the possible $\gamma$'s. The number of such $\gamma$'s is at most $m(m-1)^{\ell-1}\p^\ell(x,y)$ and the probability that $\gamma \cap \Ecal_2 \neq \emptyset$ is precisely $p_2^{\ell}(1-(p_1/p_2)^\ell)$. Since $(p_2(m-1))^\ell = 1+o(1)$ when $\ell \le \tmix$ by \eqref{e:assump1}, and since $(1-(p_1/p_2)^\ell)\leq \vep \ell$ we bound
\begin{equation} \label{e.onearmbound2}
	\P^A(\Mcal_2(x,y,\ell)) \leq C \vep \ell \p^\ell(x,y) \, .
\end{equation}
To bound the first term in the sum on the right-hand side of \eqref{e.onearmboundevent} we condition on the open and closed edges that determine $B_0(k+t)$ using the same approach as in the proof of Theorem \ref{thm:subvolume} above. Afterwards we condition on the open and closed edges of $B_0(k)$ and proceed similarly. We get
 \begin{equation}\label{e.onearmbound1}
\begin{split}
	\P^A(\Mcal_1(x,y,k,t)) &\leq \P^A (0 \stackrel{=k}{\lrfill} x) \sup_{B_1 \subset \Ecal} \P_p(x \stackrel{=t}{\lrfill} y \off B_1) \\ & \qquad  \times \sup_{B_2 \subset \Ecal} \P_p(\partial B_y(r-k-t) \neq \emptyset \off B_2).
\end{split}
\end{equation}
We now separate into four cases, corresponding to whether $t \ge \tmix$ or not, and whether $r-k-t \ge \vep^{-1}$ or not.
\medskip

The first case we consider is when $t\geq \mnot$ and $r-k-t \geq \vep^{-1}$.
In this case we may use the induction hypothesis on the last term of \eqref{e.onearmbound1}, and \eqref{e:unifconnprob} together with Theorem \ref{thm:subvolume} to bound the second term on the right-hand side of \eqref{e.onearmbound1}, yielding
\[
	 \P^A(\Mcal_1(x,y,k,t)) \leq C V^{-1} \vep (1-\vep)^{r-k} \Pp^A(0 \stackrel{=k}{\lrfill} x) \, .
\]
This together with \eqref{e.onearmbound2} gives that the sum in \eqref{e.onearmboundevent} when $t \geq \mnot$ and $r-k-t \geq \vep^{-1}$ is at most
\[
	 \sum_{x,y} \sum_{\substack{k,t \geq \mnot, \ell \leq \mnot \\ k+t \leq r-\vep^{-1}}} C V^{-1} \vep^2 (1-\vep)^{r-k} \ell \Pp^A(0 \stackrel{=k}{\lrfill} x) \p^\ell(x,y) \, .
\]
Since $\sum_y \p^{\ell}(x,y)=1$, we may sum the term $\P(0 \stackrel{=k}{\lrfill} x)$ over $x$ and bound it by $C(1-\vep)^k$ using Theorem \ref{thm:subvolume}. This yields
\begin{equation}\label{e:smallell1}
	\supa \sum_{x,y} \sum_{\substack{k,t \geq \mnot, \ell \leq \mnot \\ k+t \leq r-\vep^{-1}}} \P^A(\Mcal_1(x,y,k,t)) \P^A(\Mcal_2(x,y,\ell)) \leq C \vep(1-\vep)^r {r^2 \vep \mnot^2 \over V} \, .
\end{equation}
By \eqref{e:assump0} and since $\vep_m \gg V^{-1/3}$ we get that $r^2\vep \mnot^2/V \le (V\vep^3)^{-1/3} = o(1)$, as required.
\medskip

The second case is when $t\geq \mnot$ and $r-k-t \leq \vep^{-1}$. In this case we bound the last term of \eqref{e.onearmbound1} using \eqref{e:critonearm} (instead of the induction hypothesis) and then proceed as in the previous case to obtain
\[
	 \P^A (\Mcal_1(x,y,k,t)) \leq C V^{-1} (1-\vep)^t (r-k-t)^{-1} \Pp^A (0 \stackrel{=k}{\lrfill} x) \, .
\]
So the sum in \eqref{e.onearmboundevent} when $t \geq \mnot$ and $r-k-t \leq \vep^{-1}$ is at most
\[
	 \sum_{x,y} \sum_{\substack{k,t \geq \mnot, \ell \leq \mnot \\ k+t \geq r-\vep^{-1}}} C V^{-1} (1-\vep)^t (r-k-t)^{-1} \Pp^A(0 \stackrel{=k}{\lrfill} x) \vep \ell \p^\ell(x,y) \, .
\]
As before we use $\sum_y \p^\ell(x,y) =1$ and apply Theorem \ref{thm:subvolume} to $\sum_x \Pp^A(0 \stackrel{=k}{\lrfill} x)$ to obtain
\begin{equation}\label{e:smallell2}
	\supa \sum_{x,y} \sum_{\substack{k,t \geq \mnot, \ell \leq \mnot \\ k+t \geq r-\vep^{-1}}} \P^A(\Mcal_1(x,y,k,t)) \P^A(\Mcal_2(x,y,\ell)) \leq C \vep (1-\vep)^r {r \log(\vep^{-1}) \mnot^2 \over V} \, .
\end{equation}
By \eqref{e:assump0} and since $r \leq V^{1/3}$, we get that $V^{-1} r \log(\vep^{-1}) \mnot^2 = o(1)$ as required.
\medskip

The third case is when $t \leq \mnot$ and $r-k-t \geq \vep^{-1}$. We proceed from \eqref{e.onearmbound1}. Since $t\leq \mnot$ we may enumerate the paths connecting $x$ to $y$ in the same manner that we reached \eqref{e.onearmbound2} to get that $\Pp^{B_1}(x \stackrel{=t}{\lrfill} y) \leq C(1-\vep)^t \p^t(x,y)$ (we dropped the requirement that the paths avoid $B_1$ for an upper bound). We now proceed as in the first case, and get that the sum in \eqref{e.onearmboundevent} when $t \leq \mnot$ and $r-k-t \geq \vep^{-1}$ is at most
\[
	 \sum_{x,y} \sum_{\substack{k,t \leq \mnot, \ell \leq \mnot \\ k+t \leq r-\vep^{-1}}} C \Pp^A \big(0 \stackrel{=k}{\lrfill} x \big) \vep^2 (1-\vep)^{r-k} \ell \p^t(x,y) \p^\ell(x,y) \, .
\]
By Lemma~\ref{lem:heatkernelbound}, summing $\ell \p^t(x,y) \p^\ell(x,y)$ over $\ell, t,y$ gives a factor $O(\alpha_m/\log V)$. We also apply Theorem \ref{thm:subvolume} to $\sum_x \Pp^A(0 \stackrel{=k}{\lrfill} x)$ and obtain
\begin{equation}\label{e:smallell3}
	\supa \sum_{x,y} \sum_{\substack{k,t \leq \mnot, \ell \leq \mnot \\ k+t \leq r-\vep^{-1}}} \P^A(\Mcal_1(x,y,k,t)) \P^A(\Mcal_2(x,y,\ell)) \leq C \vep (1-\vep)^r {\vep r \alpha_m \over \log V} \, .
\end{equation}
Since $r = O(\vep^{-1} \log(\vep^3 V))$ and $\alpha_m=o(1)$ we get that that last factor is $o(1)$, as required.
\medskip

The fourth and final case is when $t \leq \mnot$ and $r-k-t \leq \vep^{-1}$. We proceed from \eqref{e.onearmbound1} and use \eqref{e:critonearm}, and then proceed exactly as in the first case. We get that the sum in \eqref{e.onearmboundevent} over $t \leq \mnot$ and $r-k-t \leq \vep^{-1}$ is at most
\[
	 \sum_{x,y}\sum_{\substack{k,t \leq \mnot, \ell \leq \mnot \\ k+t \geq r-\vep^{-1}}} C (r-k-t)^{-1} \Pp^A \big(0 \stackrel{=k}{\lrfill} x \big) \vep (1-\vep)^t \ell \p^t(x,y) \p^\ell(x,y) \, .
\]
We start by summing over $x$ to get a factor $(1-\vep)^k$, and then bound the product $(1-\vep)^{k+t}$ by $C(1-\vep)^r$ since $r-k-t \leq \vep^{-1}$. We then sum $(r-k-t)^{-1}$ over $k$ and get a factor $\log(\vep^{-1})$. We finish by summing over $\ell,t,y$ using Lemma~\ref{lem:heatkernelbound} to get
\begin{equation}\label{e:smallell4}
	\supa \sum_{x,y}\sum_{\substack{k,t \leq \mnot, \ell \leq \mnot \\ k+t \geq r-\vep^{-1}}} \P^A(\Mcal_1(x,y,k,t)) \P^A(\Mcal_2(x,y,\ell)) \le C \vep (1-\vep)^r {\alpha_m \log(\vep^{-1}) \over \log V} \, ,
\end{equation}
and the proof is completed since $\log(\vep^{-1}) \leq \log V$ and $\alpha_m=o(1)$. \qed

\subsection{Proof of the upper bounds in Theorem~\ref{thm:generalmaxdiam}(b)} These follow directly from Theorems~\ref{thm:subvolume} and \ref{thm:subdiam}.\qed

\subsection{The one-arm probability off a set}

As we have seen several times before, since $\partial B(r) \neq \emptyset$ is not monotone, it is not a priori clear that the probability of this event could not increase if we restrict ourselves to a subgraph. We believe that the unrestricted setting maximizes the one-arm probability, but we are unable to prove this. The following estimate (which we shall use several times later on) shows that as long as we do not remove too many edges, the probability does not change much. In what follows, for a subset of edges $A$ we write $\Vcal(A)$ for the set of vertices which are touched by $A$.

\begin{theorem}\label{thm:subdiamoffA}
Assume the setting of Theorem \ref{thm:generalmaxdiam}.
There exists $C>0$ such that for all integers $r$ satisfying $r\geq \vep^{-1}$ and $r = O(\vep^{-1} \log(\vep^3 V))$, and for all sets of edges $A \subset \Ecal$ satisfying $|\Vcal(A)|=O(\vep^{-2} \log(\vep^3 V))$ we have
\emph{
 \[
 	\sum_x \Pp \big(\partial B_x(r) \neq \emptyset \off A \text{ but not } \partial B_x (r) \neq \emptyset \big) = o\big(V \P_p(\dbr \neq \emptyset)\big) \, .
\]}
\end{theorem}
\proof
 If $\partial B_{x}(r) \neq \emptyset$ occurs off $A$ but $\partial B_{{x}}(r) \neq \emptyset$ does not occur, then there exists a $p$-open path $\eta$ of length $r$ in $G\setminus A = (\Vcal, \Ecal \setminus A)$ {started at $x$} such that any other $p$-open path in $G \setminus A$ between the endpoints of $\eta$ has length at least $r$ \emph{but} there exists a $p$-open path in $G$ between the endpoints of $\eta$ of length less than $r$. This path clearly needs to pass through a vertex in $\Vcal(A)$ and thus a shortcut is made. By now we are familiar with various techniques of bounding such events. For a fixed path $\eta$ of length $r$ {started at $x$}, denote by $\A_r({x,}\eta; A)$ the event that $\eta$ is the lexicographically first $p$-open shortest path in $G\setminus A$ {started at $x$}, so that $\uplus_\eta \A_r({x,}\eta; A) = \{\partial B_{{x}}(r) \neq \emptyset \off A\}$.

The event $\A_r({x,}\eta; A) \cap \{\partial B_{{x}}(r) = \emptyset\}$ implies that there exists $u,v \in\Vcal(\eta)$ and $a\in \Vcal(A)$ such that the event
\[
	 \A_r({x,}\eta; A) \circ \{ u \lrfill a \} \circ \{ a \lrfill v \}
\]
occurs.
Indeed, the set of witness edges for the first event are the edges of $\eta$ together with all the closed edges (which determine that $\eta$ is the first $p$-open shortest path). The sets of witness edges for the second and third event are the edges on two disjoint paths, connecting $a$ to $u$ and $a$ to $v$, respectively (such paths must exists because there exists a shortcut to $\eta$ passing through $\Vcal(A)$).

\label{page:Hdef}
We again split the event according to whether these shortcuts are longer than $\mnot$ or not. {Denote by $\Hcal^{\ge} (x,\eta,a; A)$} the event that {$\Acal_r(x,\eta; A)$ occurs, and that} both disjoint connections from $a$ to $u$ and $a$ to $v$ have length at least $\mnot$, and {$\Hcal^{\le}({x,}\eta,a; A)$ analogously, except that now one of the connections} has length at most $\mnot$, that is,
\[
	  \Hcal^{\ge}({x,}\eta, a; A) := \A_r({x,} \eta; A) \circ \{ u \stackrel{\geq \mnot}{\lrfill} a \} \circ \{ a \stackrel{\geq \mnot}{\lrfill} v \} \, ,
\]
and
\[
	 \Hcal^{\le}({x,}\eta, a; A) := {\big(} \A_r({x,}\eta; A) \circ \{ u \stackrel{\leq \mnot}{\lrfill} a \} \circ \{ a \stackrel{}{\lrfill} v \} {\big)} \cup {\big (} \A_r({x,}\eta; A) \circ \{ u \stackrel{}{\lrfill} a \} \circ \{ a \stackrel{\leq \mnot}{\lrfill} v \} {\big)} \, .
\]
\smallskip

We bound the probability of $\Hcal^{\ge}({x,}\eta,a; A)$ using the BKR-inequality, \eqref{e:chibds}, and \eqref{e:unifconnprob2},
\[\begin{split}
	 \sum_{a \in \Vcal(A)} \P_p(\Hcal^{\ge}({x,}\eta,a; A)) &\leq \sum_{a \in \Vcal(A), u,v\in \Vcal(\eta)} C \P_p(\A_r({x,}\eta; A)) \frac{\chi(p)^2}{V^2} \\
	 & \leq C|\Vcal(A)| r^2   \P_p(\A_r({x,}\eta; A))  \vep^{-2} V^{-2} \, .
\end{split}
\]
Hence, by our assumptions on $|\Vcal(A)|$ and $r$ we get
\[\begin{split}
	 \sum_\eta \sum_{a \in \Vcal(A)} \P_p(\Hcal^{\ge} ({x,}\eta,a;A))
	 &\leq \frac{C  \log^3(\vep^3 V) }{(V \vep^3)^{2}}  \P_p(\dbr \neq \emptyset \off A)\\
	 & = o(\P_p(\dbr \neq \emptyset)) \, ,
\end{split}\]
where the last bound is due to Lemma \ref{lem:subdiam}(a) and Theorem \ref{thm:subdiam}. \medskip

To bound the probability of $\uplus_\eta \cup_{a \in \Vcal(A)} \Hcal^{{\le}}({x,} \eta,a; A)$ we consider a further two cases: either the disjoint paths from $a$ to $u$ and $a$ to $v$ are both of length at most $\mnot$, or one of these paths has length at most $\mnot$ and the other is longer than $\mnot$. For fixed $a \in \Vcal(A)$, the union over $\eta$ of the first case implies that there exists vertices $u,v \in \Vcal(\eta)$ and integers $k,t \leq r$ such that
\begin{multline*}
	\big\{\{x \stackrel{=k}{\lrfill} u\}  \cap \{u \stackrel{=t}{\lrfill} v  \off B_x(k)\} \cap \{\partial B_v(r-k-t) \neq \emptyset \off B_x(k+t)\} \big\}\\
	\circ\big\{u \stackrel{\le \mnot}{\lrfill} a\big\} \circ \big\{a \stackrel{\le \mnot}{\lrfill} v\big\}
\end{multline*}
occurs.
As usual, the witness edges for the first disjointly occurring event are the edges of $\eta$ and all the closed edges, and the other two sets of witness edges are simply the disjoint open paths between $a$ and $u$ and between $a$ and $v$. The analysis now proceeds similarly to the proof of previous theorems in this section, splitting the sum into four parts, according to whether $t \ge \tmix$ or not, and whether $r-k-t \ge \vep^{-1}$ or not.
\medskip

We start with the case $t \geq \mnot$ and $r-k-t \ge \vep^{-1}$. We use the BKR-inequality and as before, we condition on $B_x (k+t)$ and use Theorem~\ref{thm:subdiam} to get a factor $\vep (1-\vep)^{r-k-t}$ for the probability of $\partial B_v(r-k-t) \neq \emptyset \off B_x(k+t)$. We proceed by conditioning on $B_x (k)$ and use Theorem \ref{thm:subvolume} and \eqref{e:unifconnprob2} to get a factor $V^{-1} (1-\vep)^t$. We then sum the probability of the third disjoint event over $v$ to get a factor $\mnot$ by \eqref{e:critvolball}. Then we sum the probability of the first event in the first disjoint event over $x$ to get a factor $(1-\vep)^k$ by Theorem~\ref{thm:subvolume}. Lastly, we sum the probability of the second disjoint event over $u$ to get a factor $\mnot$, again by \eqref{e:critvolball}. All this gives the bound
\[
	 \sum_{a\in \Vcal(A)} \sum_{k,t \leq r} C V^{-1} \vep (1-\vep)^{r} \mnot^2  \leq C V^{-1} r^2 \vep^{-2} \log(\vep^3 V) \mnot^2 \vep(1-\vep)^r \, .
\]
{By \eqref{e:assump0},} our assumption on $r$ and since $\vep \gg V^{-1/3}$, this quantity is $o(V(1-\vep)^r/r)$. Lemma \ref{lem:subdiam}(a) now gives the claimed bound in this case.
\medskip

The case where $t \ge \mnot$ and $r-k-t < \vep^{-1}$ is very similar, except that we now use \eqref{e:critonearm} to get a factor $(r-k-t)^{-1}$. This gives the bound
\[
	\sum_{a\in \Vcal(A)} C V^{-1} (1-\vep)^r \mnot^2 \sum_{\substack{k,t \leq r:\\ r-k-t \leq \vep^{-1}}}(r-k-t)^{-1} \leq CV^{-1} \vep^{-2} \log(\vep^3 V) \mnot^2 r \log(\vep^{-1}) (1-\vep)^r \, ,
\]
which is $o(V r^{-1} (1-\vep)^r)$ by \eqref{e:assump0} and our usual assumptions on $r$ and $\vep$.
\medskip

For the case $t \leq \mnot$ and $r-k-t \geq \vep^{-1}$ we again apply the same method of conditioning on $B_x(k+t)$ and $B_x(k)$ to get a factor $C\vep(1-\vep)^{r-k}$. At this point, instead of summing over $u$, we sum over $x$, using Theorem \ref{thm:subvolume} to bound
\[
	 \sum_{\substack{a\in \Vcal(A), u,v:\\
	 a \neq u, u \neq v, v \neq a}} \sum_{\substack{ t \le \mnot, k :\\ r-k-t \geq \vep^{-1}} } C \vep (1-\vep)^r \P_p \big(a \stackrel{\le \mnot}{\lrfill} u \big) \P_p \big(u \stackrel{=t}{\lrfill} v \big) \P_p \big(a \stackrel{\le \mnot}{\lrfill} v \big)
\]
(the restriction $a \neq u, u \neq v, v \neq a$ follows since by construction $a, u$, and $v$ must be distinct vertices).
Summing over $k$ gives a factor $r$. By enumerating over paths, in the same way we derived \eqref{e.onearmbound2}, when $s \le \mnot$ we may bound $\P_p(a \stackrel{=s}{\lrfill} u) \leq (1+o(1))\p^s(a,u)$. This yields
\[
	C \vep (1-\vep)^r r \vep^{-2} \log(\vep^3 V) \sum_{u,v} \sum_{\substack{t_1, t_2, t_3 : \\t_1+t_2+t_3 \geq3}}^{\mnot} \p^{t_1}(0,u)\p^{t_2}(u,v)\p^{t_3}(v,0),
\]
where the sum $t_1 + t_2 + t_3 \geq 3$ since $a,u,$ and $v$ are distinct vertices. By \eqref{e:assump2} and the rest of our assumptions we get this sum is again $o(V r^{-1} (1-\vep)^r)$.
\medskip

The fourth and final case is when $t \leq \mnot$ but $r-k-t \leq \vep^{-1}$. We take the same first steps as in the previous case, but when we sum over $x$ we now get the bound
\[
	 C (1-\vep)^{r-\vep^{-1}} \vep^{-2} \log(\vep^3 V) \sum_k (r-k)^{-1} \sum_{u,v} \sum_{\substack{t_1, t_2, t_3 : \\t_1+t_2+t_3 \geq3}}^{\mnot} \p^{t_1}(0,u)\p^{t_2}(u,v)\p^{t_3}(v,0) \, ,
\]
where we used \eqref{e:critonearm}.
This is at most
\[
	(1-\vep)^r \log(r) \vep^{-2} \log(\vep^3 V) \frac{ \alpha_m}{\log V} = o(V r^{-1} (1-\vep)^r)
\]
by \eqref{e:assump2} and the rest of our assumptions, and since $\log r \leq \log V$.
\qed

\section{The component of maximal diameter}\label{sec:maxdiam}
In this section we prove that $\dmax = (1+o(1))\vep^{-1} \log(\vep^3 V)$ with high probability. To start, we need a refinement of \eqref{e:smallCdiam}.

\begin{lemma}\label{lem:diametervolume}
Assume the setting of Theorem \ref{thm:general1}.
Then there exist $C < \infty$ and $c>0$ such that for any any $\alpha, \beta > 0$ satisfying $\beta \leq \alpha^2/(32 C_2)$ where $C_2$ is the constant from \eqref{e:critonearm}, and any $r \ge 4(\alpha \vee 1) \vep^{-1}$, we have
\[
	\Pp \big (\partial B(r) \neq \emptyset \and B( \alpha \vep^{-1}) \leq \beta \vep^{-2} \big ) \leq C \frac{\vep}{\alpha} (1-\vep)^{r-\alpha \vep^{-1}} e^{-c\alpha^2 /\beta}   .
\]
\end{lemma}
\begin{proof} The proof is very similar to \cite[Lemma 6.2]{NacPer08}, however, minor changes are required so we briefly repeat it here for completeness. Put $h=4\beta \vep^{-1}/\alpha$. We say that a level $j \in [\alpha \vep^{-1}/2,\alpha \vep^{-1}]$ is \emph{thin} if $|\partial B(j)|\leq h$. Define $j_1$ to be the first thin level larger than $\alpha \vep^{-1}/2$ and recursively define for $i \geq 2$,
\[
	 j_i = \min\big\{ j \geq j_{i-1} + 2C_2 h : |\partial B(j)|\leq h \big \}
\]
where $C_2$ is the constant from \eqref{e:critonearm}. We say that level $j$ is \emph{good} if there exists a vertex $w \in \partial B(j)$ such that $\partial B_w(2C_2h) \neq \emptyset \off B(j)$. By \eqref{e:critonearm} and the union bound we have that
\[
	 \P_p \big (\text{level } j \text{ is good} \,\, \big | \,\, B(j), j \text{ is thin} \big ) \leq {1 \over 2} \, .
\]
We iterate this and get that for any $n$ we have
\begin{equation}\label{diamvol.iterate} \P_p \big ( \text {level } j_i \text{ is good for all } i \leq n \,\, \big | \,\, B(\alpha \vep^{-1}/2) \big ) \leq 2^{-n} \, .\end{equation}

Now, if the {events} $\partial B(r) \neq \emptyset$ and $B( \alpha \vep^{-1}) \leq \beta \vep^{-2}$ {occur,} then the following occurs:
\begin{enumerate}
\item $\partial B(\alpha \vep^{-1}/2) \neq \emptyset$, and
\item levels $j_1, j_2, \ldots, j_n$ are good with $n$ satisfying $n \geq \alpha \vep^{-1} / (8C_2 h)$, and
\item there exists $w \in \partial B(j_n)$ such that $\partial B_w(r-j_n) \neq \emptyset$ off $B(j_n)$.
\end{enumerate}
Only (ii) requires an explanation: since $B( \alpha \vep^{-1}) \leq
\beta \vep^{-2}$ we get that at least $\alpha \vep^{-1}/4$ levels $j$ in $[\alpha \vep^{-1}/2,\alpha \vep^{-1}]$ must be thin and therefore we can find at least $\alpha \vep^{-1}/ (4 \cdot 2C_2 h)$ thin levels such that each is separated from the others by $2C_2 h$ levels. Note that we used $\alpha \vep^{-1} / (8C_2 h) \geq 1$ which follows from our assumption $\beta \leq \alpha^2/(32 C_2)$.

By \eqref{e:critonearm} the probability of (i) is at most $C \vep / \alpha$. By Theorem \ref{thm:subdiam}, the fact that level $j_n$ is thin, $j_n \le \alpha \vep^{-1}$, and the union bound, we get that
\[
	 \P_p \big ( \exists w \in \partial B(j_n) \,: \, \partial B_w(r-j_n) \neq \emptyset \off B(j_n) \,\, \big | \,\, j_n \leq 2\vep^{-1}, B(j_n) \big ) \leq C h \vep (1-\vep)^{r-\alpha \vep^{-1}} \, .
\]
Combining these with \eqref{diamvol.iterate} and plugging in the value of $h$ gives that
\[
	\P\big (\partial B(r) \neq \emptyset \and B( \alpha \vep^{-1}) \leq \beta \vep^{-2} \big ) \leq C \frac{\vep}{\alpha} (1-\vep)^{r-\alpha \vep^{-1}} e^{-c\alpha^2 /\beta}  \, ,
\]
where $c=(32C_2)^{-1}$ and we used again our assumption $\beta \leq \alpha^2/(32 C_2)$. \end{proof}

\subsection{Proof of the upper bound in Theorem \ref{thm:generalmaxdiam}(a).}
We begin by proving the upper bound on $\dmax$, that is, we will prove that under the conditions of the theorem, for any $\delta>0$ we have
\begin{equation}\label{pf:generalmaxdiam.upper} \P \big ( \exists \Ccal \with \diam(\Ccal) \geq (1+\delta) \vep^{-1} \log(\vep^3 V) \big ) = o(1)  \, .\end{equation}

Put $r=(1+\delta) \vep^{-1} \log(\vep^3 V)$. The initial idea is that if there is a vertex $x$ such that $\partial B_x(r) \neq \emptyset$, then $|B_x (\vep^{-1})|$ is typically of order $\vep^{-2}$ and so there are in fact $\vep^{-2}$ vertices $u$ with $\partial B_u(r-\vep^{-1} ) \neq \emptyset$, allowing us to use Markov's inequality. However, with some small probability the $\vep^{-1}$-ball will have $o(\vep^{-2})$ vertices, invalidating the argument. We fix this with a multi-scale argument using Lemma \ref{lem:diametervolume}.

This simple idea works rather easily when $|B_x (\vep^{-1})|$ is smaller than $\vep^{-2} /\log(\vep^{-1})$. Indeed, by Lemma \ref{lem:diametervolume} and the union bound we get that for any $\beta>0$
\[
	 \P \big (\exists x \with \partial B_x(r) \neq \emptyset \and |B_x (\vep^{-1})| \leq \beta \vep^{-2} / \log(\vep^{-1})\big ) \leq  {C \vep V (1-\vep)^r e^{-c\log(\vep^{-1})/\beta}} \, .
\]
We have that $(1-\vep)^r \leq (\vep^3 V)^{-1-\delta}$. We put $\beta = c/2$ so that $e^{-c\log(\vep^{-1})/\beta} = \vep^2$, and hence the above probability is at most $(\vep^3 V)^{-\delta} = o(1)$.

If $\log(\vep^{-1}) \leq (\vep^3 V)^{\delta/2}$, then we may conclude, since by the triangle inequality the event
\[
	 \big \{ \exists x \with \partial B_x(r) \neq \emptyset \and |B_x(\vep^{-1})| \geq \beta \vep^{-2}/ \log(\vep^{-1}) \big \}
\]
implies that there are at least $\beta \vep^{-2} /\log(\vep^{-1})$ vertices $u$ such that $\partial B_u (r-\vep^{-1}) \neq \emptyset$. By Theorem \ref{thm:subdiam} and Markov's inequality we get that this probability is at most
\[
	 {C V \vep (1-\vep)^r \log(\vep^{-1})  \over  \beta \vep^{-2} } \leq {\log(\vep^{-1}) \over  (\vep^3 V)^{\delta}} = o(1) \, .
\]

If on the other hand $\log(\vep^{-1}) > (\vep^3 V)^{\delta/2}$, we define $N=N(\vep)$ to be
\[
	 N = \min \,\, \big \{ n \geq 2: \log^{(n)}(\vep^{-1}) \leq (\vep^3 V)^{\delta/2} \big \} \, ,
\]
where $\log^{(n)}$ is the composition of $\log$ with itself {$n-1$} times. Define an increasing sequence of radii $(r_k)_{k=1}^N$ by
\[
	 r_k := \vep^{-1} + {\vep^{-1} \over \log^{(k+1)}(\vep^{-1})} \, .
\]
If there exists a vertex $x$ such that $\partial B_x(r) \neq \emptyset$ and $|B_x (\vep^{-1})| \geq \beta \vep^{-2} /\log(\vep^{-1})$ {both occur,} then one of the following events must occur:
\begin{enumerate}
\item
\[
	|B_x (r_{N})| \geq {\vep^{-2} \over (\log^{(N)}(\vep^{-1}))^2} \, , \quad \text{or}
\]
\item there exists $k\in\{2,\ldots, N\}$ such that
\[
	 |B_x (r_k)| \leq {\vep^{-2} \over (\log^{(k)}(\vep^{-1}))^2} \qquad \and  \qquad |B_x (r_{k-1})| \geq {\vep^{-2} \over (\log^{(k-1)}(\vep^{-1}))^2} \, , \quad \text{or}
\]
\item
\[
	 |B_x (r_1)| \leq {\vep^{-2} \over (\log(\vep^{-1}))^2} \qquad \and \qquad |B_x (\vep^{-1})| \geq {\beta \vep^{-2}  \over \log(\vep^{-1})}\, .
\]
\end{enumerate}

By the triangle inequality and since $r_k \leq 2\vep ^{-1}$, if (i) occurs, then there are at least $\vep^{-2}/(\log^{(N)}(\vep^{-1}))^2$ vertices $u$ such that $\partial B_u (r-2\vep^{-1}) \neq \emptyset$. As before, Theorem~\ref{thm:subdiam} together with Markov's inequality gives that the probability of this is at most
\[
	 {C V \vep (1-\vep)^{r-2\vep^{-1}} (\log^{(N)}(\vep^{-1}))^2 \over \vep^{-2}} \leq {C (\log^{(N)}(\vep^{-1}))^2 \over (\vep^3 V)^\delta} = o(1) \, ,
\]
by definition of $N$.

If (ii) occurs for some $k\in\{2,\ldots, N\}$, then each vertex $u \in B_x (r_{k-1})$ {satisfies}
\[
	|B_u (r_k-r_{k-1})|\leq \frac{\vep^{-2}}{(\log^{(k)}(\vep^{-1}))^2} \qquad \text{ and } \qquad \partial B_u(r - 2 \vep^{-1}) \neq \emptyset.
\]
By Lemma \ref{lem:diametervolume}, for each vertex $u$ the probability of this is at most
\[
	 {C \vep (1-\vep)^{r-2\vep^{-1}} \exp \Bigg(-c {(\log^{(k)}(\vep^{-1}))^2 \over (\log^{(k+1)}(\vep^{-1}))^2}} \Bigg) \leq {C \vep (\vep^3 V)^{-1-\delta} (\log^{(k-1)}(\vep^{-1}))^{-3}}  \, ,
\]
where the last inequality follows from our usual assignment of variables and the fact that $\log^{(k)}(\vep^{-1})/(\log^{(k+1)}(\vep^{-1}))^2 \to \infty$. Since $|B_x (r_{k-1})| \geq \vep^{-2} / (\log^{(k-1)}(\vep^{-1}))^2$ we get by Markov's inequality that the probability that (ii) occurs for some $k \in \{2,\ldots, N\}$ is at most
\[
	 {C (\vep^3 V)^{-\delta} \big(\log^{(k-1)}(\vep^{-1})\big)^{-1}} \, ,
\]
and summing over $k$ gives that the probability of (ii) tends to $0$ as well.

The bound on (iii) is performed in the same way, using Lemma~\ref{lem:diametervolume}. This concludes the proof of \eqref{pf:generalmaxdiam.upper}. \qed
\medskip

\label{page:lbdiam}
\subsection{Proof of the lower bound in Theorem \ref{thm:generalmaxdiam}(a).}\label{page:lbdiam}
Let us now prove the lower bound on $\dmax$, i.e., that for any $\delta>0$ we have
\begin{equation}\label{pf:generalmaxdiam.lower} \P \big ( \exists \Ccal \with \diam(\Ccal) \geq (1-\delta) \vep^{-1} \log(\vep^3 V) \big ) = 1-o(1)  \, .\end{equation}

Let $\delta>0$ be arbitrary and put $r=(1-\delta)\vep^{-1}\log(\vep^3 V)$. Let $D_r$ denote the random variable
\begin{equation}\label{e:Drdef}
	 D_r = \# \big\{ v : \partial B_v(r) \neq \emptyset \and |\Ccal(v)| \leq 5 \vep^{-2} \log(\vep^3 V) \big\} \, ,
\end{equation}
so that it suffices prove that $D_r > 0$ with probability tending to $1$. We prove this using a second moment argument. By \eqref{e:uppertail} and \eqref{e:chibds} it follows that
\begin{equation}\label{e:Clarge}
 \P_p \big( |\Ccal(v)| \geq 5 \vep^{-2} \log(\vep^3 V) \big) \leq C \vep (\vep^3 V)^{-2} = o(r^{-1} (1-\vep)^r) \, ,
\end{equation}
by our choice of $r$. Hence by Lemma \ref{lem:subdiam}(a) we have that
\[
	\P_p\big(\dbr \neq \emptyset \and |\Ccal| \leq 5 \vep^{-2} \log(\vep^3 V)\big) \geq (1-o(1)) \P_p(\dbr \neq \emptyset).
\]
By Lemma \ref{lem:subdiam}(a) again we get that for some fixed $c>0$,
\begin{equation}\label{pf:diamfirstmom} \E_p [D_r] \geq (1-o(1)) V \P_p(\dbr \neq\emptyset) \geq  c V r^{-1} (1-\vep)^r \geq c \vep^{-2} (\vep^3 V)^\delta \, .\end{equation}

Any pair of vertices $u$ and $v$ counted in $D_r$ can either belong to the same component or not. Thus, the second moment of $D_r$ can be bounded by
\[\begin{split}
	\E_p [D_r^2] & \leq \sum_{u,v} \P_p\big (\partial B_u(r) \neq \emptyset,  v \in \Ccal(u),  |\Ccal(u)| \leq 5\vep^{-2}\log(\vep^3 V) \big ) \\
	& \quad +\sum_{u,v} \P_p\big (\partial B_u(r) \neq \emptyset, v \not \in \Ccal(u),  |\Ccal(u)| \leq 5\vep^{-2} \log(\vep^3 V) , \partial B_v(r) \neq \emptyset \big ) \, .
\end{split}
\]
Denote the first sum by (I) and the second by (II). Bounding the first term is easy:
\[\begin{split}
	\hbox{{\rm (I)}} & = V \E_p\big[ |\Ccal| \indi_{\{\dbr \neq \emptyset\}} \indi_{\{|\Ccal| \leq 5 \vep^{-2} \log(\vep^3 V)\}}\big]\\
	 & \leq 5 V\vep^{-2} \log(\vep^3 V) \P_p(\dbr \neq \emptyset) \\
	 & \leq C V \vep^{-1}\log(\vep^3 V)(1-\vep)^r \\
	 & \leq C \vep^{-4} (\vep^3 V)^\delta \log(\vep^3 V) \, ,\end{split}
\]
where the one before last inequality is due to Theorem \ref{thm:subdiam}, and the last inequality comes from plugging in the value of $r$. Since $ \vep^3 V \to \infty$, by \eqref{pf:diamfirstmom} we deduce that (I) is $o(\E_p [D_r]^2)$.

To estimate (II) we condition on $\Ccal(u)$ such that $\partial B_u(r) \neq \emptyset$, and then require that $\partial B_v(r)\neq \emptyset$ occurs off $\Ccal(u)$. We write this as
\begin{equation}\label{e:twosum}
	 \hbox{{\rm (II)}} = \sum_{u} \sum_{\substack{A \subset G :\\
	  {|\Vcal(A)|} \leq 5\vep^{-2} \log(\vep^3 V), \\ \partial B^{{A}}_u(r) \neq \emptyset}} \P_p(\Ccal(u) = A) \sum_v \P_p(\partial B_v(r) \neq \emptyset \off {\Vcal(A)}) \, .
\end{equation}
%\Tim{To verify that $\partial B_u(r) \neq \emptyset$ we need to condition on subgraphs.}
Such {subgraphs} $A$ satisfy the condition of Theorem \ref{thm:subdiamoffA} {that $|\Vcal(A)| = O(\vep^{-2} \log (\vep^3 V))$}, so we bound
\[
	 \sum_v \P_p(\partial B_v(r) \neq \emptyset \off {\Vcal(A)}) \leq (1+o(1)) V \P_p(\partial B_v(r) \neq \emptyset).
\]
Applying this bound and summing \eqref{e:twosum} over $A$ and $u$ gives
\[
	\text{(II)} \le (1+o(1))V^2 \Pp(\partial B(r) \neq \emptyset)^2.
\]
Comparing with \eqref{pf:diamfirstmom} we get that (II)$=(1+o(1))\E [D_r]^2$, which, together with our previous estimate, implies that
\begin{equation} \label{e.tight2ndmom} \E_p [D_r^2] = (1+o(1))\E [D_r]^2 \, . \end{equation}
The proof is now completed using the inequality $\P(Z>0) \geq \E [Z]^2/\E[ Z^2]$, valid for any non-negative random variable~$Z$. \qed
\medskip

\section{The component with the largest mixing time}\label{sec:mixing}

\subsection{Proof of the upper bound in Theorem \ref{thm:generalmaxdiam}(c)}
The upper bound follows from the lemma below, which is proved in \cite{NacPer08}.

\begin{lemma}[Corollary 4.2 from \cite{NacPer08}]\label{lem:commute} Let $G=(\Vcal,\Ecal)$ be a connected graph. The mixing time of a lazy simple random walk on $G$ satisfies
\[
	\Tmix(G,1/4) \le 8 |\Ecal| \diam(G).
\]
\end{lemma}

We know from Theorems \ref{thm:general1}(a) and \ref{thm:generalmaxdiam}(a) that for all clusters $\Ccal$ at $p = p_c(1-\vep)$ with $\vep =o(1)$ and $\vep^3 V \to \infty$, for all $\delta>0$ we have that $|\Ccal| \le (2+\delta)\vep^{-2} \log (\vep^3 V)$ and $\diam(\Ccal) \le (1+\delta) \vep^{-1} \log(\vep^3 V)$ with high probability. {This does not, however, directly imply a good estimate of the maximal number of edges in a cluster, which is what we need.}
The following lemma gives such an estimate, and the proof of the upper bound in Theorem \ref{thm:mixing} then follows.

\begin{lemma}\label{lem:dense} Assume the setting of Theorem \ref{thm:general1} and let $\Ecal_{\mathrm{max}}$ denote the number of edges of the component with {the} maximal number of edges. Then
\[
	\Pp \big(\Ecal_{\mathrm{max}} \ge 9 \vep^{-2} \log(\vep^3 V)\big) = o(1).
\]
\end{lemma}
\begin{proof}
Fix $\delta>0$ and write $M = 3 \vep^{-2} \log(\vep^3 V)$. We bound
\[
	\Pp \big(\Ecal_{\mathrm{max}} \ge 9 \vep^{-2} \log(\vep^3 V) \big)  \le \Pp(|\Cmax| \ge   M )  + \Pp\big(\Ecal_{\mathrm{max}} \ge  3M \, , \, |\Cmax| \leq M \big) \, .
\]
By Theorem \ref{thm:general1}(a) and our choice of $M$ the first term on the right-hand side of the above is $o(1)$ and it remains to show the second term is also $o(1)$.

To that aim, given a vertex $x$, we write $\Ecal_x$ for the number of edges of the connected component containing $x$. Conditioned on the vertex set of $\Ccal(x)$ and on a spanning tree of $\Ccal(x)$ that consists {only of} open edges (such a spanning tree could be, for instance, the BFS tree of $\Ccal(x)$) we have that $\Ecal_x - |\Ccal(x)|$ is stochastically dominated by a Binomial random variable with parameters $m|\Ccal(x)|$ and $p$. Thus, if $|\Ccal(x)|\leq M$, the probability of the event that $\Ecal_x \geq 3M$ is bounded above by a probability that the value of a Binomial $(mM,p)$ random variable exceeds $3M$. We use the standard Chernoff bound \cite[Theorem~2.1]{JanLucRuc00} that if $X \sim \mathsf{Bin}(n,q)$, then
\[
	 \P(X \ge nq + t) \le \exp\big(-t^2/(2nq + 2t/3)\big) \, ,
\]
for any $t>0$. Since $p=m^{-1}(1+o(1))$ we obtain that
\begin{equation}\label{e.usechernoff} \Pp \big ( |\Ccal(x)| \leq M \and \Ecal_x \geq 3M \big ) \leq e^{-c \vep^{-2} \log (\vep^3 V) } \, ,\end{equation}
for some universal $c>0$. It is straightforward to see that by our assumption on $\vep$ the latter quantity is $o(V^{-1})$ and so the probability that there exists such vertex $x$ is $o(1)$, concluding our proof.
\end{proof}

\subsection{Proof of the lower bound in Theorem \ref{thm:generalmaxdiam}(c)}

For the proof of the lower bound we use a lemma from \cite{NacPer08} for which we require some definitions:

\begin{enumerate}
\item For integer $\rad$ and vertex $v$ we call an edge $e$ a \emph{lane for $(v,\rad)$} if $e$ is an edge between $\partial B_v (j-1)$ and $\partial B_v (j)$ for some $0 < j < \rad$, and there exists an open path with first edge $e$ from $\partial B_v(j-1) $ to $\partial B_v(\rad)$ that does not pass through $\partial B_v(j-1)$.

\item For integers $\rad,j$ and $\ell$ with $0 < j < \rad$ we say that \emph{level $j$ has $\ell$ lanes for $(v,\rad)$} if there are at least $\ell$ edges between $\partial B_v(j-1)$ and $\partial B_v(j)$ that are a lane for $(v,\rad)$.

\item We say that \emph{$v$ is $\ell$-lane rich for $(k,\rad)$} if more than half the levels $j \in [k/2,k]$ have at least $\ell$ lanes for $(v,\rad)$.
\end{enumerate}

\begin{lemma}[Lemma 5.4 from \cite{NacPer08}]\label{lem:Tmixlb}
Let $G=(\Vcal,\Ecal)$ be a graph and $v \in \Vcal$. Suppose that $q,h,k,\rad$ and $\ell$ are positive integers satisfying:
\begin{itemize}
	\item $|B_v(h)| \ge q$, and
	\item $v$ is \emph{not} $\ell$-lane rich for $(k,\rad)$, and
	\item $|E(B_v(\rad))| < \tfrac13 E(\Ccal(v))$, and
	\item $h < k/(4\ell)$.
\end{itemize}
Then
\[
	\Tmix(G) \ge \frac{q k}{12 \ell}.
\]
\end{lemma}

Thus, our goal is to choose the parameters of the above lemma appropriately and to show that a vertex $v$ satisfying the assumptions of the lemma above exists. We fix a positive sequence $(\omega_m)$ such that $\omega_m=o(1)$ and
\begin{equation}\label{chooseomega} \omega_m^{{7}} \gg {1 \over \log(\vep^3 V)} \qquad \and \qquad \omega_m^2 \gg \alpha_m \, ,\end{equation}
where $\alpha_m$ is the sequence given in the statement of Theorem \ref{thm:generalmaxdiam}. We also fix some $\delta>0$ and set our parameters accordingly by:

\begin{equation}
\begin{split} r &= (1-\delta)\vep^{-1} \log(\vep^3 V) \, , \qquad  \rad  = \omega_m r \, , \qquad  k = \omega^2_m r \, , \\ h &= \omega^3_m r \, , \qquad
q = \omega^7_m \vep^{-2} \log(\vep^3 V) \, , \qquad \ell  = 1/(8\omega_m) \, .
\end{split} \label{mixingsetparameters}
\end{equation}

\begin{lemma}\label{lem:A123}
Assume the setting of Theorem \ref{thm:generalmaxdiam} and consider the choice of parameters in \eqref{chooseomega} and \eqref{mixingsetparameters}.
Then
\begin{enumerate}
\item[(a)] $\displaystyle \Pp \big(|B_v(h)| < q \text{\emph{ and }} \partial B_v(r) \neq \emptyset\big) = o\big(\Pp(\partial B_v(r) \neq \emptyset)\big)$.

\item[(b)] $\displaystyle \Pp \big(v \text{\emph{ is $\ell$-lane rich for $(k,\rad)$ and }} \partial B_v(r) \neq \emptyset \big) = o\big(\Pp(\partial B_v(r) \neq \emptyset)\big)$.

\item[(c)]	$\displaystyle \Pp \big(E(B_v(\rad)) > |\Ccal(v)|/3 \text{\emph{ and }} \partial B_v(r) \neq \emptyset \big) = o\big(\Pp(\partial B_v(r) \neq \emptyset)\big)$.
\end{enumerate}
\end{lemma}

\noindent See Figure \ref{fig:Pf65} for a sketch of these three events.
\begin{figure}
	\includegraphics[width = .8\textwidth]{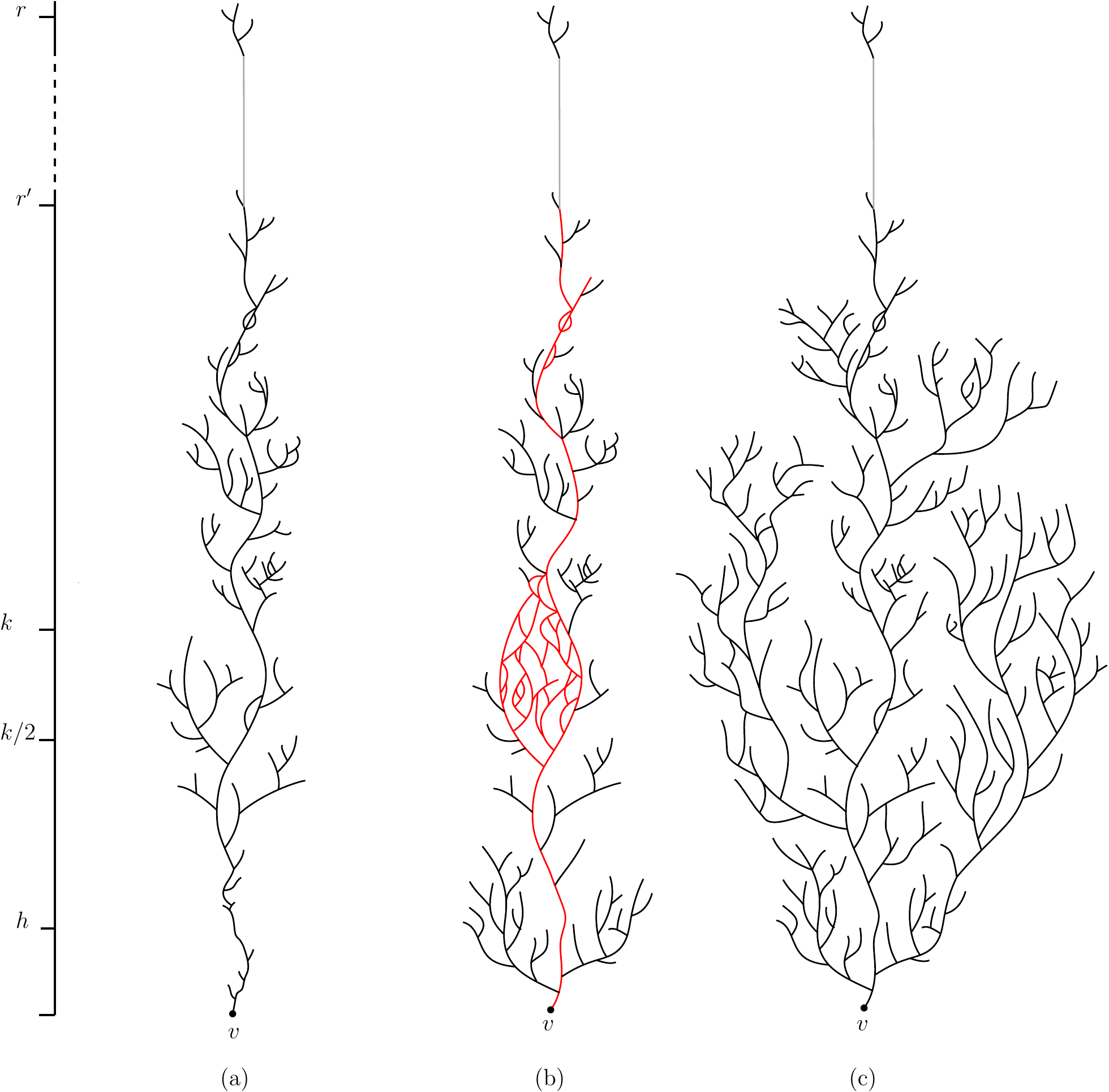}
	\caption{\label{fig:Pf65} A sketch of the three events. In (a) the tree is very skinny up to height $h$. In (b) there are many lanes between $k/2$ and $k$ for $(v,\rad)$ (the lanes have been colored red). In (c) the tree is very fat up to height $\rad$.}
\end{figure}

\proof
(a) This follows by Lemmas \ref{lem:subdiam} and \ref{lem:diametervolume} and our choice of $h,q$ and $r$. Indeed, applying Lemma \ref{lem:diametervolume} with $\alpha = \omega_m^3 (1-\delta) \log(\vep^3 V)$ and $\beta = \omega_m^7 \log(\vep^3 V)$, and observing that the conditions of the lemma are met for $m$ sufficiently large, we obtain the bound
\[
	\Pp \big(|B_v(h)| < q \text{ and } \partial B_v(r) \neq \emptyset\big) \le \frac{C\vep}{\log(\vep^3 V)} (1-\vep)^{r} \frac{(\vep^3 V)^{(\omega_m^3-\frac{c}{ \omega_m})(1-\delta)}}{\omega_m^3}.
\]
Since $\omega_m = o(1)$ and $\vep^3 V \to \infty$, the claim now follows by Lemma \ref{lem:subdiam}.
\medskip

(b) Let $L$ denote the number of lanes between levels $k/2$ and $k$. If $v$ is $\ell$-lane rich for $(k,\rad)$, then $L \ge \ell k /4$, so by Markov's inequality,
\[
	\Pp(v \text{ is $\ell$-lane rich for }(k,\rad) \text{ and } \partial B_v(r) \neq \emptyset) \le \frac{\Ep[L \indi_{\{ \partial B_v(r) \neq \emptyset\}}]}{\ell k /4}.
\]
The claim follows if we prove that $\Ep[L \indi_{\{\partial B_v(r) \neq \emptyset\}}] \le C k \Pp(\partial B_v(r) \neq \emptyset)$, since $1/\ell = o(1)$.

Recall from \eqref{e:Agammadef} and \eqref{gammar} that ${\Acal_r (v,\eta)}$ is the event that $\eta$ is the first $p$-open path of length $r$ emanating from ${v}$, and that $\uplus_\eta {\Acal_r (v,\eta) = \{\partial B_v(r) \neq \emptyset\}}$. We condition on $\eta$:
\[
	\Ep[L \indi_{\{\partial B_v(r) \neq \emptyset\}}] = \sum_\eta \Ep[L \mid {\Acal_r (v,\eta)}] \Pp({\Acal_r (v,\eta)}).
\]
Conditioned on ${\Acal_r (v,\eta)}$, any edge that is a lane for $(v,\rad)$ can either belong to $\eta$, or be on a path extending from $\eta$ to $\partial B_v(\rad)$ without intersecting $\eta$ again, or be on a path starting from a vertex of $\eta$ and ending in a different vertex of $\eta$. More precisely, if ${\Acal_r (v,\eta)}$ happens and $e = \{\ule,\ole\}$ is a lane such that $\ule \in \partial B_v(j-1)$ and $\ole \in \partial B_v(j)$ for some $j \in [k/2,k]$, then one of the following must occur:
\begin{enumerate}
	\item $e$ is an edge of $\eta$, or
	\item there exists $s \in [0, k]$ and $t \in [1,k]$ such that $\big\{\eta(s) \stackrel{=t}{\lrfill} \ule \text{ off }\eta \cup \{e\} \big\}$ {and } $\{e \hbox{ is open}\}$ { and }$\big\{\partial B_{\ole} (\rad-s-t-1) \neq \emptyset \text{ off }\eta \cup \{e\} \cup B_{\eta(s)}(t) \big\}$
		occurs, or
	\item there exist $s \in [0,k]$, $t \in [s+1, r]$, and a $p$-open path $\gamma$ that is edge-wise disjoint from $\eta$, with $\gamma(0) = \eta(s)$, $\gamma(|\gamma|) = \eta(t)$, and $e \in \gamma$.
\end{enumerate}
\smallskip

We now bound the contributions to $\Ep[L \indi_{\{\partial B_v(r) \neq \emptyset\}}]$ from summing over the edge $e$ of each of these cases separately, showing all three cases contribute at most order $k\Pp(\partial B_v(r) \neq \emptyset)$.
\medskip

Case (i) is easy and, conditioned on ${\Acal_r (v,\eta)}$, contributes precisely the edges of $\eta$ between levels $k/2$ and $k$, i.e., precisely $k/2$ edges to $L$.
\medskip

To bound the contribution of case (ii) conditioned on ${\Acal_r (v,\eta)}$, we {again condition} on the ball $B_{\eta(s)}(t)${, and proceed as before. Some care is required, because this case has a new subtlety: both $\Acal_r (v,\eta)$} and the {events of case} (ii) are not monotone events (with respect to adding edges) and {there may not exist} disjoint witnesses for their occurrence (in particular, the closed edges that determine that $\eta$ is the first shortest path may be needed to {determine} the shortest connection between $\eta(s)$ and $\ole$). {Thus} we cannot appeal to BKR-inequality. We have to use the {off-method and the attendant conditioning scheme (described in Section \ref{sec:offmethod})} with care{:}

We condition on ${\Acal_r (v,\eta)}$ and on all the closed edges {for} which at least one endpoint is {part} of $B_v(r)$. Because the events in (ii) are all ``off $\eta$'',
{conditioning on $\Acal_r (v,\eta)$ does not affect the events in (ii). The conditioning only affects (ii) through the the closed edges of the conditioning,} which we can simply add to the set of edges that the events in (ii) are ``off'' of. This way we may use our usual conditioning scheme and use Theorems~\ref{thm:subvolume} and \ref{thm:subdiam}.
{Doing so, we find that the} last event in (ii) contributes a factor $\vep (1-\vep)^{\rad-s-t-1}$ by Theorem \ref{thm:subdiam}, {and} the second event contributes a factor $p=(1+o(1))m^{-1}$ by \eqref{e:pcestimate}. We sum the probabilities of the first event over $\ule$ and get a factor $(1-\vep)^{t}$ by Theorem \ref{thm:subvolume}. We then sum over $\ole,t,s$ and get a contribution of order at most $k(1-\vep)^{r-k}=o(k)$ from case (ii).
\medskip

The estimate of the contribution of case (iii) is more involved, but very similar to the proof of Theorem \ref{thm:subdiam}. Analogous to $\Fcal^{\le}(\eta)$ and $\Fcal^{\ge}(\eta)$ defined in the course of the proof of Theorem \ref{thm:subdiam},
define the events $\Jcal^{\le}(\eta, e)$ and $\Jcal^{\ge}(\eta, e)$ to be the events that case (iii) occurs and that $|\gamma| \le 2\tmix+1$ or $|\gamma| \ge 2\tmix+1$, respectively. We bound these two events separately, starting with $\Jcal^{\ge}(\eta, e)$, {for} which we bound its probability conditioned on ${\Acal_r (v,\eta)}$:

If $\Jcal^{\ge}(\eta, e)$ occurs, then either the part of $\gamma$ leading to $\ule$ is longer than $\mnot$, or the part of $\gamma$ starting from $\ole$ is longer than $\mnot$. Thus by the BKR-inequality (where, as usual, the witnesses to ${\Acal_r (v,\eta)}$ are the open edges of $\eta$ together with closed edges, and the other two are the corresponding open paths) we get
\[\begin{split}
	\sum_{e} \Pp \big(\Jcal^{\ge}(\eta, e) \mid {\Acal_r (v,\eta)} \big) \le &  \sum_{s = 0}^k \sum_{t=s+1}^r \sum_{e}p \Big(\Pp^{\eta} \big(\eta(s) \stackrel{\ge \tmix}{\lrfill} \ule \big) \Pp^{\eta}\big(\ole \conn \eta(t) \big)\\
	& \qquad \qquad \quad + \Pp^{\eta} \big(\eta(s) \conn \ule \big) \Pp^{\eta}\big(\ole \stackrel{\ge \tmix}{\lrfill}  \eta(t) \big) \Big) \\
	 \le & C k r \vep^{-2} V^{-1},
\end{split}
\]
where the second bound follows from \eqref{e:unifconnprob2} and \eqref{e:chibds}. Since $r \vep^{-2} V^{-1} = o(1)$, the contribution from $\Jcal^{\ge}(\eta, e)$ is $o(k)$.

The contribution of $\Jcal^{\le}(\eta, e)$ to $\Ep[L \indi_{\{\partial B_v(r) \neq \emptyset\}}]$ is bounded differently. We write $L$ as a sum of indicators over the edge $e$, and for each edge separately we take the union of $\eta$ of the events $\Jcal^{\le}(\eta, e)$. The event $\uplus_\eta \Jcal^{\le}(\eta, e)$ implies that there exist integers $s, t, \ell$, with $s + t \le r$ and $s \leq k$ and $\ell \le 2 \tmix $, and vertices $x, y$ such that the following events occur disjointly:
\[\begin{split}
	\Mcal_1(x,y,s,t) &:=  \big\{ 0 \stackrel{=s}{\lrfill} x \big\} \cap \big\{x \stackrel{=t}{\lrfill} y \off B_0(s) \big\} \cap \big\{\partial B_y (r-t-s) \neq \emptyset \off B_0(t) \big\},\\
	\Mcal_2(x,y,e,l) & := \big\{ \exists \gamma \, : \,  \gamma \text{ is a $p$-open path, } |\gamma| = \ell , \gamma(0) = x, \gamma(\ell) =y, e \in \gamma \big\}.
\end{split}\]
Applying the BKR-inequality yields
\[
	\sum_{e} \Pp \big(\uplus_\eta \Jcal^{\le}(\eta, e) \big) \le \sum_{x,y,e} \sum_{\substack{ s \le k, s+t \leq r, \\ l \le 2 \tmix+1}} \Pp(\Mcal_1(x,y,s,t)) \Pp(\Mcal_2(x,y,e,\ell)).
\]

To bound the probability of $\sum_{e} \Pp(\Mcal_2(x,y,e,\ell))$ we first apply the union bound to $\gamma$. To this end, write $\Gamma_\ell (x,y)$ for the set of all {simple} paths of length $\ell$ from $x$ to $y$. We bound
\[
	\sum_{e} \Pp(\Mcal_2(x,y,e,\ell)) \le \sum_{\gamma \in \Gamma_\ell(x,y)} \sum_{e} \Pp(\gamma \text{ open, } e \in \gamma) \le \ell |\Gamma_\ell (x,y)| p^\ell,
\]
where the factor $\ell$ is due to the fact that any fixed $\gamma \in \Gamma_\ell(x,y)$ contains $\ell$ edges, so the sum over $e$ contains exactly $\ell$ non-zero terms whose value is $p^\ell$. We bound $|\Gamma_\ell(x,y)|$ by $m (m-1)^{\ell-1} \mathbf{p}^\ell(x,y)$ as usual. By \eqref{e:assump1} we have that $(p(m-1))^\ell = O(1)$ for $\ell \le 2 \tmix$, so we obtain
\[
	\sum_{e} \Pp(\Mcal_2(x,y,e,\ell)) \le C \ell \mathbf{p}^\ell(x,y).
\]
Compare this with the bound in \eqref{e.onearmbound2} and note that the current bound is a factor $\vep^{-1}$ bigger.

The rest of the analysis is now performed exactly as the analysis of four cases {of} $\uplus_\eta \Fcal^{\le}(\eta)$ in last part of the proof of Theorem \ref{thm:subdiam} (starting with \eqref{e.onearmbound1}). Deriving the four bounds analogous to \eqref{e:smallell1}--\eqref{e:smallell4}, we get
\[
	\sum_{e} \Pp \big(\uplus_\eta \Jcal^{\le}(\eta, e) \big) \le C  \vep (1-\vep)^r \left(\frac{k r \tmix^2}{V} + \frac{k \vep^{-1} \log(\vep^{-1}) \tmix^2}{V} + \frac{k \alpha_m}{\log V} + \frac{\alpha_m \vep^{-1} \log(\vep^{-1})}{\log V} \right).
\]
We make two remarks about this derivation:
(1) we need $\omega_m^2 \log(\vep^3 V) \to \infty$, because the proof requires that $k \ge \vep^{-1}$, and
(2) it follows immediately from \cite[proof of Theorem~4.5]{HofNac12} that Lemma~\ref{lem:heatkernelbound} remains valid upon replacing $\tmix$ by $2 \tmix$.

The lower and upper bounds from Lemma~\ref{lem:subdiam}(a) and Theorem~\ref{thm:subdiam} differ by a factor $\log(\vep^3 V)$ for our choice of $R$, so the desired bound follows if each of the four factors on the right-hand side is $o(k / \log(\vep^3 V))$. The first error term satisfies this bound by our choice of $r$ in \eqref{mixingsetparameters} and by {\eqref{e:assump0}}, and similarly for the second term. The third term is bounded likewise simply because $\alpha_m=o(1)$. The fourth satisfies the required bound since we assumed $\omega_m^2 \gg \alpha_m$.
\medskip

Combining the contributions due to (i), (ii), and (iii), we obtain
\[
	\Ep[L \indi_{\{\partial B_v(r) \neq \emptyset\}}] = {(1+o(1))k \over 2} \Pp(\partial B_v(r) \neq \emptyset),
\]
as desired. This completes the proof of (b).
\medskip

(c)
Let $M = c' \vep^{-2} \log(\vep^3 V)$, where $c'>0$ is a small constant that will be chosen soon. If $E(B(\rad)) \geq \tfrac13 |\Ccal|$ occurs, then either $|\Ccal| \leq 3M$, or $|\Ccal|\geq 3M$ and $E(B(\rad)) \leq M$. Thus we bound
\[\begin{split}
	\Pp\big(E(B(\rad)) \geq \tfrac13 |\Ccal| , \partial B(r) \neq \emptyset\big)
		&\le \Pp\big(|\Ccal| \leq 3 M , \partial B(r) \neq \emptyset\big) \\
		&\quad + \Pp\big(E(B(\rad)) \geq M , |B(\rad)| \leq \tfrac{1}{10} M\big) \\
		&\quad + \Pp\big(|B(\rad)| \geq \tfrac{1}{10} M , \partial B(r) \neq \emptyset\big).
\end{split}\]

We now show each term is $o(\Pp(\partial B(r) \neq \emptyset))$. We first choose $c'>0$ small enough so that by \eqref{e:smallCdiam} and Lemma \ref{lem:subdiam}(a) we get that the first term is $o(\Pp(\partial B(r) \neq \emptyset))$. The second term is bounded by $o(V^{-1})$ as in \eqref{e.usechernoff}, which is much smaller than $\Pp(\partial B(r) \neq \emptyset)$ by our choice of $r$ and Lemma \ref{lem:subdiam}(a). For the third term we use a similar proof strategy as in Lemma \ref{lem:secmom}. Using Markov's inequality we bound
\[
	\Pp\big(|B(\rad)| > \tfrac{1}{10} M, \partial B(r) \neq \emptyset\big) \le \frac{10 \Ep[|B(\rad)| \indi_{\{\partial B(r) \neq \emptyset\}}]}{M}.
\]
As usual we define ${\Acal_r (0,\eta)}$ as the event that a simple path $\eta$ of length $r$ is the first open shortest path of length $r$ {started at $0$,} and write
\[
	\Ep[|B(\rad)| \indi_{\{\partial B(r) \neq \emptyset\}}] = \sum_\eta \sum_x \Pp\big(0 \stackrel{\le \rad}{\lrfill} x \text{ and } {\Acal_r (0,\eta)}\big).
\]
If $\big\{0 \stackrel{\le \rad}{\lrfill} x \big\}$ and ${\Acal_r (0,\eta)}$ occur, then there must exist an integer $j \in [0,\rad]$ such that ${\Acal_r (0,\eta)} \circ \{\eta(j) \conn x\}$ occurs. Applying the BKR inequality and summing over $x$ (using \eqref{e:chibds}) and then $\eta$ gives
\[
	\Ep[|B(\rad)| \indi_{\{\partial B(r) \neq \emptyset\}}] \le \sum_{j=0}^{\rad} \Pp(\partial B(r) \neq \emptyset) \chi(p) \le (1 + o(1)) \vep^{-1} \rad \Pp(\partial B(r) \neq \emptyset).
\]
As a result,
\[
	\Pp\big(|B(\rad)| > \tfrac{1}{10} M, \partial B(r) \neq \emptyset\big) \le \frac{10 \vep^{-1} \rad}{M} \Pp(\partial B(r) \neq \emptyset).
\]
By our choices of $\rad$ and $M$, and since $\omega_m = o(1)$, this bound is also $o(\Pp(\partial B(r) \neq \emptyset))$, completing the proof of (c).\qed
\medskip

\proof[Proof of the lower bound of Theorem \ref{thm:generalmaxdiam}(c)]
Let $\omega_m, h,k,q,\rad, r,$ and $\ell$ be the parameters chosen in \eqref{chooseomega} and \eqref{mixingsetparameters}. Note that to prove the required lower bound for any $\omega_m=o(1)$ we can assume without loss of generality that $\omega_m$ satisfies \eqref{chooseomega}. Let $L_r$ denote the number of vertices $v$ satisfying
\begin{itemize}
    \item $\partial B_v(r) \neq \emptyset$, and
	\item $|B_v(h)| \ge q$, and
	\item $v$ is \emph{not} $\ell$-lane rich for $(k,\rad)$, and
	\item $|E(B_v(\rad))| < \tfrac13 E(\Ccal(v))$, and
	\item $|\Ccal(v)| \le 5 \vep^{-2} \log(\vep^3 V)$.
\end{itemize}
Also recall the definition of $D_r$ from \eqref{e:Drdef}. {By definition $L_r \leq D_r$. By Lemma~\ref{lem:A123} and \eqref{e:Clarge}} it follows that $\E_p[L_r] = (1-o(1))\E_p[D_r]$. Furthermore, in \eqref{e.tight2ndmom} it is proved that
$\Ep[D_{r}^2] = (1+o(1))\Ep[D_{r}]^2$ and so it follows that $\E_p[L_r^2] = (1+o(1))\E_p[L_r]^2$. By \eqref{pf:diamfirstmom} it follows that $\E_p[L_r] \to \infty$ and so we conclude that with high probability $L_r \to \infty$ and in particular there exists at least one cluster that satisfies the conditions of Lemma \ref{lem:Tmixlb}. That is, with high probability,
\[
	\Tmix(\Ccal^\star) \ge \frac{q k}{12 \ell} \ge  c \omega_m^{10} \vep^{-3} \log (\vep^3 V)^2.
\]
Since the choice of $\omega_m$ was arbitrary, this completes the proof. \qed

\subsection*{Acknowledgments}
TH is supported by the Netherlands Organisation for Scientific Research (NWO) through the Gravitation {\sc Networks} grant 024.002.003. AN is supported by ISF grant 1207/15, and ERC starting grant 676970.

\begin{small}
\bibliographystyle{abbrv}
\bibliography{TimsBib}
\end{small}
\end{document}